\documentclass[10 pt,reqno]{cedram-aif}
\usepackage{amssymb,amsmath,amstext,amsgen,amsfonts,amsthm,verbatim,hyperref,graphicx,float}
\usepackage{amsrefs}

\newcommand{\N}{\mathbb{N}}
\newcommand{\R}{\mathbb{R}}

\newcommand{\Z}{\mathbb{Z}}
\newcommand{\E}{\mathbb{E}}
\renewcommand{\H}{\mathbb{H}}

\renewcommand{\epsilon}{\varepsilon}
\renewcommand{\phi}{\varphi}

\equalenv{corollary}{coro}
\equalenv{proposition}{prop}
\equalenv{remark}{rema}

\begin{document}

\title[Markov convexity of the Heisenberg group]{Markov convexity and nonembeddability of the Heisenberg group}
\alttitle{Convexit\'e Markov et non-plongeabilit\'e du groupe de Heisenberg}
\author{\firstname{Sean} \lastname{Li}}
\email{seanli@math.uchicago.edu}

\address{Department of Mathematics\\The University of Chicago\\Chicago, IL 60637}
\altkeywords{groupe de Heisenberg, convexit\'e Markov, biLipschitz, plongements}

\thanks{I am grateful to Assaf Naor for suggesting this avenue of research to me as well as many helpful conversations.  I'd also like to thank Max Engelstein for helpful conversations and the referee for pointing out many mistakes.  The research presented here is supported in part by NSF postdoctoral research fellowship DMS-1303910.}

\begin{abstract}
  We show that the continuous infinite dimensional Heisenberg group $\mathbb{H}_\infty$ is Markov 4-convex and that the 3-dimensional Heisenberg group $\mathbb{H}_1$ (and thus also $\mathbb{H}_\infty$) cannot be Markov $p$-convex for any $p < 4$.  As Markov convexity is biLipschitz invariant and Hilbert spaces are Markov 2-convex, this gives a different proof of the classical theorem of Pansu and Semmes that the Heisenberg group does not biLipschitz embed into any Euclidean space.
  
  The Markov convexity lower bound follows from exhibiting an explicit embedding of Laakso graphs $G_n$ into $\mathbb{H}_\infty$ that has distortion at most $C n^{1/4} \sqrt{\log n}$.  We use this to derive a quantitative lower bound for the biLipschitz distortion of balls of the discrete Heisenberg group into Markov $p$-convex metric spaces.  Finally, we show surprisingly that Markov 4-convexity does not give the optimal distortion for embeddings of binary trees $B_m$ into $\mathbb{H}_\infty$ by showing that the distortion is on the order of $\sqrt{\log m}$.
\end{abstract}

\begin{altabstract}
  Nous montrons que le groupe de Heisenberg de dimension infinie continue $\mathbb{H}_\infty$ est Markov 4-convexe et que le groupe de Heisenberg 3 dimensions $\mathbb{H}_1$ (et donc aussi $\mathbb{H}_\infty$) ne peut pas \^{e}tre Markov $p$-convexe pour tout $p < 4$.  Comme convexit\'e Markov est bilipschitzienne pr\`es invariant et espaces de Hilbert sont Markov 2-convexe, ce qui donne une autre preuve du th\'eor\`eme classique de Pansu et Semmes que le groupe de Heisenberg ne bilipschitzienne pr\`es int\'egrer dans un espace euclidien.
  
  La convexit\'e Markov limit inf\'erieure d\'ecoule pr\'esentant un plongement explicite de graphes de Laakso $G_n$ en $\mathbb{H}_\infty$ qui pr\'esente une distorsion au plus $C n^{1/4} \sqrt{\log n}$.  Nous l'utilisons pour calculer une quantitative limite inf\'erieure pour la distorsion bilipschitzienne pr\`es de boules du groupe de Heisenberg discr\`ete dans espaces m\'etriques de Markov $p$-convexit\'e.  Enfin, nous montrons de mani\`ere surprenante que Markov 4-convexit\'e ne donne pas le distorsion optimale pour plongements d'arbres binaires $B_m$ en $\mathbb{H}_\infty$, en montrant que la distosion est de l'ordre de $\sqrt{\log m}$.
\end{altabstract}

\maketitle

\section{Introduction}

A Banach space $X$ is said to be finitely representable in another Banach space $Y$ if there exists $K \geq 1$ so that for every finite dimensional subspace $Z \subset X$, there exists a finite dimensional subspace $Z' \subset Y$ so that $d_{BM}(Z,Z') \leq K$, where $d_{BM}$ is the Banach-Mazur distance.  Ribe proved in \cite{ribe} that if two Banach spaces are uniformly homeomorhic (that is, there exists $f : X \to Y$ such that $f$ and $f^{-1}$ are uniform homeomorphisms), then $X$ is finitely representable into $Y$ and vice versa.  Note that this implies that linear properties of Banach spaces that depend only on their finite dimensional substructure are preserved by maps that preserve the metric structure.  This motivated the ``Ribe program'', a research program that reformulates such linear properties in purely metric terms.  For a more details about the Ribe program, see the surveys \cite{ball,naor}.

Recall that a biLipschitz embedding $f : (X,d_X) \to (Y,d_Y)$ is said to have distortion $D \geq 1$ if there exists some $s \in (0,\infty)$ such that
\begin{align*}
  s \cdot d_X(x,y) \leq d_Y(f(x),f(y)) \leq Ds \cdot d_X(x,y), \qquad \forall x,y \in X.
\end{align*}
Given two metric spaces, we say $X$ embeds into $Y$ with distortion $D$ if there exists some biLipschitz embedding $f : X \to Y$ with distortion $D$.  We also define the following quantity:
\begin{multline*}
  c_Y(X) := \inf \{D \geq 1 : \text{there is a biLipschitz embedding } f : X \to Y \\
  \text{ of distortion } D\}.
\end{multline*}
In this paper, when we use a graph as a metric space, the only points in the space are the vertices.  The edges do not exist in the space; they just define the path metric.  We will require all edges of a single graph to have constant length although the actual length itself is irrelevant as calculating distortion allows us to rescale the metrics (by the factor $s$).

The first result in the Ribe program was by Bourgain in \cite{bourgain} where he showed that a Banach space $X$ is {\it not} superreflexive if and only if complete binary trees of depth $n$ equipped with the path metric biLipschitzly embed into $X$ with uniformly bounded distortion over $n$.  It was later shown in \cite{johnson-schechtman} that the same statement holds except with binary trees replaced by diamond graphs and Laakso graphs.  In the sequel, given metric spaces $(X,d_X)$ and $(Y,d_Y)$, we will let $c_Y(X)$ denote the infimal distortion required to biLipschitzly embed $X$ into $Y$ (it can be infinite).

For $p \in [2,\infty)$, a Banach space $X$ is said to be $p$-convex if it is uniformly convex and the modulus of convexity can be taken to be $\delta(\epsilon) = C\epsilon^p$ for some $C >0$.  Through the deep works of James \cite{james-1,james-2}, Enflo \cite{enflo}, and Pisier \cite{pisier} it is known that all superreflexive spaces can be renormed to be $p$-convex for some $p \geq 2$.  A metrical characterization of $p$-convexity didn't come until 22 years after Bourgain's result.

Given a Markov chain $\{Z_t\}_{t \in \Z}$ on some state space $\Omega$ and $s \in \Z$, we let $\{\tilde{Z}_t(s)\}_{t \in \Z}$ denote the Markov chain on $\Omega$ that equals $Z_t$ when $t \leq s$ and then evolves independently (with respect to the same transition probabilities as $Z_t$) for $t > s$.  Following \cite{LNP}, we say a metric space $(X,d_X)$ is Markov $p$-convex for some $p > 0$ if there exists $\Pi > 0$ so that for every Markov chain $\{Z_t\}_{t \in \Z}$ on $\Omega$ and every $f : \Omega \to X$, we have that
\begin{align}
  \sum_{k=0}^\infty \sum_{t \in \Z} \frac{\E\left[ d(f(Z_t),f(\tilde{Z}_t(t-2^k)))^p \right]}{2^{kp}} \leq \Pi^p \sum_{t \in \Z} \E[d(f(Z_t),f(Z_{t-1}))^p]. \label{markov-convexity-defn}
\end{align}
It was proven in \cite{MN} that a Banach space is Markov $p$-convex if and only if it can be renormed to be $p$-convex.  Thus, Markov $p$-convexity is the metrical characterization of $p$-convexity.

It was shown in \cite{li-carnot} that for each Carnot group $G$, there is some $p < \infty$ for which $G$ is Markov $p$-convex, and an explicit upper bound for the power of convexity was computed in terms of the step of the Carnot group.  However, the bound appears to be far from optimal.  For example, the upper bound it gives for the (continuous) infinite dimensional Heisenberg group $\H_\infty$ is 8.  We will improve upon this result with the following theorem.
\begin{theorem} \label{4-convex}
  $\H_\infty$ is Markov 4-convex.
\end{theorem}

We will further show that 4 is the optimal power of Markov convexity for the Heisenberg group.  We will do so by showing that a sequence of Laakso-like graphs embed into the usual three dimensional Heisenberg group $\H_1$ with small distortion.
\begin{theorem} \label{laakso-distortion}
  Laakso graphs $\{G_n\}_{n=1}^\infty$ embed into $\H_1$ with distortion $O((\log |G_n|)^{1/4} \sqrt{\log \log |G_n|})$.
\end{theorem}

This good embedding of Laakso graphs will translate to a lower bound for the power of Markov convexity for $\H_1$.
\begin{corollary} \label{optimal-convex}
  $\H_1$ is not Markov $p$-convex for any $p < 4$.
\end{corollary}

As $\H_1$ admits a biLipschitz embedding into $\H_\infty$, we get that Theorem \ref{laakso-distortion} and Corollary \ref{optimal-convex} also hold for $\H_\infty$ and Theorem \ref{4-convex} holds for $\H_1$.

An immediate further corollary of this corollary is that the Heisenberg group $\H_1$ does not biLipschitz embed into any metric space that is Markov $p$-convex for any $p < 4$.  As Hilbert spaces are metric spaces that are uniformly 2-convex and so Markov 2-convex, we get a new proof of the Pansu-Semmes theorem that the Heisenberg group does not embed into any Euclidean space \cite{pansu,semmes}.

Theorem \ref{4-convex} and Corollary \ref{optimal-convex} say that what 2 is for Euclidean space in terms of the Pythagorean theorem, 4 is the natural analogue in the Heisenberg group.  A similar phenomenon occurs in the context of the analyst's traveling salesman problem in the Heisenberg group where, again, the natural power to look at is 4 \cite{li-schul-1,li-schul-2} while the natural analogue in Euclidean space is 2.

We can use this embedding of Laakso graphs to show the following lower bound for distortion of balls of the discrete Heisenberg group into metric spaces that are highly convex.
\begin{corollary} \label{discrete-nonembedding}
  Let $B(n)$ denote balls of radius $n$ of the discrete Heisenberg group $\H(\Z)$ and $X$ a metric space that is Markov $p$-convex.  Then there exists some $C > 0$ so that
  \begin{align}
    c_X(B(n)) \geq C \frac{(\log n)^{\frac{1}{p}-\frac{1}{4}}}{\sqrt{\log \log n}}. \label{quantitative-nonembed}
  \end{align}
\end{corollary}

This result should be contrasted with \cite{lafforgue-naor} where it was shown that if $X$ is a $p$-convex Banach space, then $c_X(B(n)) \gtrsim (\log n)^{1/p}$, an asymptotically sharp estimate.  That result requires that the target space be a Banach space but gives lower bounds for all powers of convexity, whereas Corollary \ref{discrete-nonembedding} holds for general metric space targets, but only gives meaningful distortion lower bounds for $p < 4$, which is to be expected given Theorem \ref{4-convex}.  Theorem 7.5 of \cite{li-carnot} shows that the target spaces in the latter case are not a subset of those in the former.  It should be noted that the Heisenberg group seems to be especially hard to embed into Banach spaces as it does not even embed into $L_1$, a Banach space that is not uniformly convex (it doesn't even have the Radon-Nikodym property) \cite{cheeger-kleiner-1,cheeger-kleiner-2,ckn}.

Until now, all known distortion bounds for embeddings of diamond/Laakso graphs $G_n$ and binary trees $B_n$ into known non-doubling Markov $p$-convex spaces---namely $p$-convex Banach spaces---have had the same asymptotics, namely $(\log |G_n|)^{1/p}$ and $(\log \log |B_n|)^{1/p}$, respectively \cite{matousek,johnson-schechtman}.   Such bounds precisely match the bounds one would get from using the Markov $p$-convexity inequality (the computation is essentially done in the proof of Corollary \ref{optimal-convex}).  We have seen that $\H_\infty$ is Markov 4-convex and the $(\log |G_n|)^{1/4}$ distortion bound still holds for Laakso graph embeddings into $\H_\infty$.  Thus, it seems reasonable to expect that the distortion of binary trees would be $(\log \log |B_n|)^{1/4}$ as suggested by 4-convexity.  However, we will show the following theorem.

\begin{theorem} \label{tree-nonembedding}
  There exists some absolute constant $C > 0$ so that any embedding of $\{B_n\}_{n=1}^\infty$ into $\H_\infty$ has distortion at least $C\sqrt{\log \log |B_n|}$.
\end{theorem}

This then means that Markov convexity does not say much about quantitative bounds on embedding of binary trees into Carnot groups, which can be thought of as the nonabelian analogues of Banach spaces.  This lower bound is sharp up to constants as $\ell_2$ embeds biLipschitzly into $\H_\infty$ and it is known that $c_{\ell_2}(B_n) \leq C\sqrt{\log \log |B_n|}$ for some other $C > 0$ \cite{bourgain}.

Clearly, such a bound cannot be derived from Markov 4-convexity of $\H_\infty$ and so we must proceed via another route.  If one looks in the literature, one finds that \cite{matousek} provides another method of computing distortion lower bounds for embedding of binary trees into $p$-convex Banach spaces.  It is this approach that we will use---with some nontrivial modifications.  We briefly describe the strategy of \cite{matousek}.  There, it was shown by metric differentiation that if $f$ is a Lipschitz embedding of a large enough binary tree into a $p$-convex Banach space, then there exists a subgraph of the tree on which $f$ sends to a $\delta$-fork (the terminology will be reviewed in Section 3).  The result of \cite{matousek} then comes from the fact that tips of $\delta$-forks in $p$-convex Banach spaces must collapse by a factor of $\delta^{1/p}$.

As is, this method does not work for $\H_\infty$ because the analogue of the fork collapse lemma for $\H_\infty$ collapses the tips of the fork by a factor of $\delta^{1/4}$, which would only give a lower bound of $(\log \log |B_n|)^{1/4}$.  However, it turns out that the tips of a $\delta$-fork in $\H_\infty$ must be in a special configuration in order to see the $\delta^{1/4}$-collapse.  Otherwise, they would see a $\delta^{1/2}$ collapse.  One can modify the metric differentiation technique of \cite{matousek} to get a large connected collection of $\delta$-forks (a $\delta$-broom if you will) and then show using the pigeonhole principle that the $\delta$-subforks associated to the $\delta$-broom cannot all be configured to see the $\delta^{1/4}$ collapse.

\subsection{Preliminaries}
The (continuous) $2n+1$ Heisenberg group of dimension $2n+1$ is the Lie group $\H_n = (\R^n \times \R^n \times \R,\cdot)$ with the group product
\begin{align*}
  (x,y,z) \cdot (x',y',z') = \left(x + x', y + y', z + z' + \frac{1}{2} (x \cdot y' - x' \cdot y) \right).
\end{align*}
One can also define the infinite dimensional Heisenberg group as $(\ell_2 \times \ell_2 \times \R,\cdot)$ where $\ell_2$ is usual real sequence space and the Hilbert inner product is used to get the same group product.  Note that $x \cdot y' - x' \cdot y$ is the canonical symplectic form.

It can be immediately verified that the Heisenberg groups are not abelian and the origin is the identity.  We call the center, which is $\{(0,0,t) : t \in \R\}$, the {\it vertical axis}.

For finite $n$, there exists a natural path metric on $\H_n$ that we will define as such.  We define $\Delta$ to be the left invariant subbundle of the tangent bundle by setting $\Delta_0$ to be the 1-codimensional plane spanned by $\R^n \times \R^n \times \{0\}$-plane and using the smoothness of the group multiplication to pushforward $\Delta_0$ to every point $x \in \H_n$.  Similarly, we can endow $\Delta$ with a left-invariant scalar product $\{\langle \cdot,\cdot \rangle_x\}_{x \in \H_n}$.  Then we can define the Carnot-Carath\'eodory metric between two points $x,y \in \H_n$ as
\begin{multline*}
  d_{cc}(x,y) := \inf \left\{ \int_a^b \langle \gamma'(t),\gamma'(t) \rangle_{\gamma(t)} dt : \gamma \in C^1([a,b];\H_n),\right. \\
  \left.\gamma(a) = x, \gamma(b) = y, \gamma'(t) \in \Delta_{\gamma(t)}\right\}.
\end{multline*}
The continuous paths $\gamma : I \to \H_n$ for which $\gamma'(t) \in \Delta_{\gamma(t)}$ are called {\it horizontal paths}.  A special case of Chow's theorem (see {\it e.g.} \cite{montgomery}) states that between two points in $\H_n$ there always exists a horizontal path and so $d_{cc}$ is a finite metric on $\H_n$.  Because we are taking the Riemannian length over a subclass of curves, this geometry is sometimes also called sub-Riemannian geometry.

It is well known that if a curve $(\gamma_x,\gamma_y,\gamma_z) : I \to \H_n$ is piecewise horizontal, then
\begin{multline}
  \gamma_z(b) - \gamma_z(a) - \frac{1}{2} \left( \gamma_x(a) \cdot \gamma_y(b) - \gamma_x(b) \cdot \gamma_y(a) \right)  \\
  = \frac{1}{2} \int_a^b \gamma_x(t) \cdot \gamma_y'(t) - \gamma_y(t) \cdot \gamma_x'(t) ~dt. \label{swept-area}
\end{multline}
When $n=1$, the vertical change in terms of group multiplication is equal to the algebra area swept by $(\gamma_x,\gamma_y)$ when viewed as a curve in $\R^2$.  On the other hand, given a curve $\gamma_0$ in $\R^2$, one can use the identity \eqref{swept-area} to lift $\gamma_0$ to a horizontal curve $\gamma$ in $\H_1$.  Notice that $\gamma$ is unique up to translation in the $z$-coordinate.

An important feature of the Heisenberg group is that for each $\lambda > 0$, there exists an automorphism
\begin{align*}
  \delta_\lambda : \H_n &\to \H_n \\
  (x,y,z) &\mapsto (\lambda x,\lambda y, \lambda^2 z)
\end{align*}
that scales the metric, {\it i.e.} $d_{cc}(\delta_\lambda(x),\delta_\lambda(y)) = \lambda d_{cc}(x,y)$.  To see this fact, one simply needs to check that the Jacobian of $\delta_\lambda$ scales $\langle \cdot, \cdot \rangle_x$ by $\lambda$.

We now introduce another metric on $\H_n$ which makes sense even for $n = \infty$.  This metric has the advantage that distances between points can be computed directly from coordinates.  Let
\begin{align*}
  N : \H_n &\to \R \\
  (x,y,t) &\mapsto \left( (|x|^2+|y|^2)^2 + t^2 \right)^{1/4}
\end{align*}
denote the Koranyi norm.  We can use it to define a left-invariant metric on $\H_n$ as
\begin{align*}
  d(x,y) := N(x^{-1}y), \qquad \forall x,y \in \H_n.
\end{align*}
It is known that this is indeed a metric \cite{cygan} ({\it i.e.} it satisfies the triangle inequality) and is biLipschitz equivalent to the Carnot-Carath\'eodory metric.  Note that $\delta_\lambda$ also scales $d$.  As all the results of this paper are given up to multiplicative constants, we see that proving the results for the Koranyi metric then proves them for the Carnot-Carath\'eodory metric also.  Thus, we will work with the Koranyi metric from now on.

Rotations of each canonical symplectic plane are isometric automorphisms of $\H_n$.  This can easily be seen by remembering that such rotations preserve the canonical symplectic form and then looking at the formulas for the Koranyi norm and group multiplication.

Let $\tilde{\pi} : \H_n \to \H_n$ denote the map $\tilde{\pi}(x,y,z) = (x,y,0)$.  It should be noted that this is {\it not} a homomorphism.  For $g \in \H_n$ define
\begin{align*}
 NH(g) = d(\tilde{\pi}(g),g).
\end{align*}
Thus, $NH$ quantifies how non-horizontal an element of $\H_n$ is by measuring its distance to the horizontal element ``below'' it.  We will also let
\begin{align*}
  \pi : \H_n &\to \R^n \times \R^n \\
  (x,y,z) &\mapsto (x,y)
\end{align*}
be the homomorphism from $\H_n$ to $\R^n \times \R^n$.  It is easily verifiable from looking at the Koranyi metric that this is 1-Lipschitz.

The discrete Heisenberg group $\H(\Z)$ is the finitely generated discrete group $\H(\Z) = (\Z^3,\cdot)$ where the group product is
\begin{align*}
  (a,b,c) \cdot (a',b',c') = (a + a', b + b', c + c' + ab').
\end{align*}
The group can be shown to be generated by the elements $(\pm 1,0,0)$ and $(0,\pm 1,0)$.  The metric on $\H(\Z)$ is then the word metric associated to this finite set of generators.  It is known that $\H_1$ embeds quasi-isometrically into $\H(\Z)$.  That is, there exist $c_0,c_1 > 0$ and a map $f : \H_1 \to \H(\Z)$ so that
\begin{align}
  \frac{1}{c_0} d(x,y) - c_1 \leq d(f(x),f(y)) \leq c_0 d(x,y) + c_1, \qquad \forall x,y \in \H. \label{quasi-isometry}
\end{align}

\section{Markov convexity and nonembeddability of $\H$}
\subsection{Upper bound}
For $a,b \in \H_\infty$, let $\frac{a+b}{2}$ denote the midpoint of the affine line segment between $a$ and $b$ when they are viewed as points in $\ell_2 \times \ell_2 \times \R$.  We can then also define $\frac{a}{2} := \frac{a+0}{2}$.  As group translations in $\H$ are affine maps of $\ell_2 \times \ell_2 \times \R$, we see that the affine midpoint between two points is preserved by the group multiplication.  It is also clearly preserved by rotations of canonical symplectic planes.  While it is true that affine midpoints are not preserved under the dilation homomorphism $\delta_\lambda$, we will never use dilation in this section.

We have the following convexity inequality for the Heisenberg group.

\begin{proposition} \label{convex}
 For $u,v,w \in \H_\infty$, we have
 \begin{align*}
  \frac{1}{2} \left(d(u,v)^4 + d(v,w)^4\right) \geq \left( \frac{d(u,w)}{2} \right)^4 + d\left(\frac{u+w}{2},v\right)^4 + 2^{-4} NH(u^{-1}w)^4
 \end{align*}
\end{proposition}

\begin{proof}
 By a translation, we may suppose $u = (0,0,0)$.  Let $v = (x,y,z)$, $w = (r,s,t)$.  We let $v' = (x,y)$ and $w' = (r,s)$ be vectors in $\ell_2 \times \ell_2$.  Let $\tau = y \cdot r - x \cdot s$.  We have by the parallelogram identity and Jensen's inequality that
 \begin{align}
   \frac{|v'|^4 + |w'-v'|^4}{2} &\geq \left( \frac{|v'|^2 + |w'-v'|^2}{2} \right)^2 \notag \\
   &= \left|v' - \frac{w'}{2} \right|^2 + \frac{1}{2^4} |w'|^2 + \frac{1}{2} |w'|^2 \left| v' - \frac{w'}{2} \right|^2 \notag \\
   &\geq \left|v' - \frac{w'}{2} \right|^2 + \frac{1}{2^4} |w'|^2 + \frac{1}{2} \tau^2. \label{e:convex-2}
 \end{align}
 In the last inequality, we used the fact that $\tau$ represents the symplectic form applied to $v'$ and $w'$ which we can bound as
 \begin{align*}
   \tau = \omega(v',w') = \omega(v'- w'/2,w') \leq \left|v' - \frac{w'}{2} \right| |w'|.
 \end{align*}
 
 We have the following
 \begin{align*}
  d(u,v)^4 &= |v'|^4 + z^2, \\
  d(v,w)^4 &= |w'-v'|^4 + \left(t - z + \frac{1}{2}\tau \right)^2, \\
  d(u,w)^4 &= |w'|^4 + t^2, \\
  d\left( \frac{u+w}{2},v \right)^4 &= \left|v' - \frac{w'}{2} \right|^4 +  \left( \frac{1}{2}t - z + \frac{1}{4} \tau \right)^2, \\
  NH(u^{-1}w)^4 &= t^2.
 \end{align*}
 Using \eqref{e:convex-2}, we see that it suffices to prove the following inequality
 \begin{align*}
   \frac{z^2 + (t - z + \tau/2)^2}{2} + \frac{1}{2} \tau^2 \geq \frac{t^2}{2^3} + \left( \frac{1}{2} t - z + \frac{1}{4} \tau \right)^2.
 \end{align*}
 By applying the parallelogram identity in $\R$ on the left hand side, we further reduce to proving the following inequality
 \begin{align*}
   \frac{1}{4} \left(t + \frac{\tau}{2}\right)^2 + \frac{1}{2} \tau^2 \geq \frac{t^2}{2^3}.
 \end{align*}
 An easy application of the parallelogram identity shows that this inequality is true.
\end{proof}

\begin{lemma} \label{shrink}
 If $u,v,w \in \H_\infty$, then
 \begin{align*}
  d(u,v)^4 \leq 32 \left( d \left( \frac{u + w}{2}, \frac{v + w}{2} \right)^4 + NH(u^{-1}w)^4 + NH(v^{-1}w)^4 \right).
 \end{align*}
\end{lemma}

\begin{proof}
 Again, by a translation, we may suppose that $w = 0$ and $u = (r,s,t)$ and $v = (x,y,z)$.  Let $\tau = s \cdot x - r \cdot y$.  Then
 \begin{align*}
  d(u,v)^4 &= (|r-x|^2 + |s-y|^2)^2 + \left( z - t + \frac{1}{2} \tau \right)^2, \\
  d\left( \frac{u}{2}, \frac{v}{2} \right)^4 &= \left( \left|\frac{r-x}{2}\right|^2 + \left|\frac{s-y}{2}\right|^2\right)^2 + \left( \frac{z - t}{2} + \frac{1}{8} \tau \right)^2, \\
  NH(u)^4 &= t^2, \\
  NH(v)^4 &= z^2.
 \end{align*}
 Note that the first term on the right hand sides of the first and second lines are multiples of each other (by a factor of $\frac{1}{16}$), so we may ignore them.  We also have that
 \begin{align*}
  2\left( t^2 + z^2 \right) \geq (z - t)^2.
 \end{align*}
 Taking this into account, it then suffices to prove that
 \begin{align*}
  \left( z - t + \frac{1}{2} \tau \right)^2 \leq 8\left( (z - t)^2 + \left( z - t + \frac{1}{4} \tau \right)^2 \right).
 \end{align*}
 Letting $a = z - t$ and $b = \tau$, we are reduced to showing that
 \begin{align}
  8 \left( a^2 + \left( a + \frac{1}{4} b \right)^2 \right) - \left( a + \frac{1}{2} b \right)^2 \geq 0.
 \end{align}
 An elementary calculus exercise shows that the left hand side of the above inequality takes a minimum value of $6a^2$.
\end{proof}

\begin{proposition} \label{4-pt-convex}
 For every $x,y,z,w \in \H_\infty$,
 \begin{align*}
  \frac{1}{2} \left( 2d(x,y)^4 + d(y,w)^4 + d(y,z)^4 \right) \geq \frac{d(x,w)^4 + d(x,z)^4}{2^4} + \frac{d(z,w)^4}{512}.
 \end{align*}
\end{proposition}

\begin{proof}
 We may first suppose that $x = 0$.  Applying Proposition \ref{convex} to the pairs $x,y,z$ and $x,y,w$, we get (still writing $x$ to keep things clear)
 \begin{align*}
  \frac{1}{2} \left(d(x,y)^4 + d(y,z)^4\right) &\geq \left( \frac{d(x,z)}{2} \right)^4 + d\left( \frac{z}{2},y\right)^4 + 2^{-4} NH(x^{-1}z)^4, \\
  \frac{1}{2} \left(d(x,y)^4 + d(y,w)^4\right) &\geq \left( \frac{d(x,w)}{2} \right)^4 + d\left( \frac{w}{2},y\right)^4 + 2^{-4} NH(x^{-1}w)^4.
 \end{align*}
 Adding the inequalities together, we have that it suffices to prove that
 \begin{multline*}
  512 \left( d\left( \frac{z}{2}, y \right)^4 + d\left( \frac{w}{2},y \right)^4 + 2^{-4} (NH(x^{-1}z)^4 + NH(x^{-1}w)^4) \right) \\
  \geq d(z,w)^4.
 \end{multline*}
 Note that
 \begin{align*}
  8 \left( d\left( \frac{z}{2},y \right)^4 + d\left( \frac{w}{2},y \right)^4 \right) \geq \left( d\left( \frac{z}{2},y \right) + d \left( \frac{w}{2},y \right) \right)^4 \geq d\left( \frac{w}{2}, \frac{z}{2} \right)^4.
 \end{align*}
 Thus, we finish the proof by appealing to Lemma \ref{shrink}.
\end{proof}

\begin{proof}[Proof of Theorem \ref{4-convex}]
  The proof of Theorem 2.1 of \cite{MN} shows that Markov 4-convexity follows directly from a four point 4-convexity inequality of the form given by Proposition \ref{4-pt-convex}.
\end{proof}

\subsection{Lower bound}
In this section, we will let $\H$ denote the three dimensional Heisenberg group $\H_1$.

\begin{figure}
\centering
  \scalebox{0.6}{
    \includegraphics{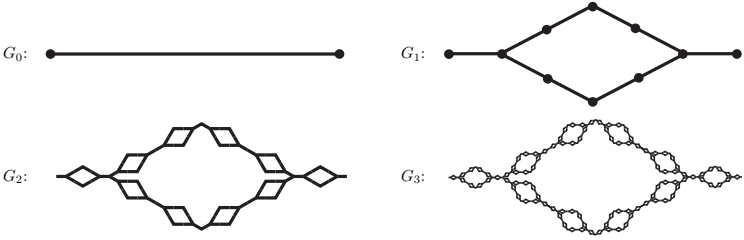}
    \setlength{\unitlength}{1mm}
    \begin{picture}(0,0)
      \put(-164,39){$G_0$:}
      \put(-77,39){$G_1$:}
      \put(-164,12){$G_2$:}
      \put(-77,12){$G_3$:}
    \end{picture}
  }
  \caption{The first four Laakso graphs \label{g:laakso}}
\end{figure}
We now prove that the Markov convexity upper bound shown in the previous subsection is tight and use it to derive Theorem \ref{laakso-distortion}.  Laakso graphs were described in \cite{laakso,lang-plaut}.  We will define the graphs $\{G_i\}_{i=0}^\infty$ as follows.  The first stage $G_0$ is simply an edge and $G_1$ is pictured as in Figure \ref{g:laakso}.  To get $G_k$ once $G_{k-1}$ is constructed, we replace each edge of $G_{k-1}$ with a copy of $G_1$.  We will choose not to rescale the metric so the diameters of $G_k$ will be $6^k$.  By abuse of terminology, we will still call these graphs Laakso graphs.

Given $G_n$, we say $G$ is an {\it unscaled copy} of $G_k$ in $G_n$ if it is isometric to $G_k$ and each edge of $G$ has length 1.  If we do not require that each edge of $G$ has length 1, then we say $G$ is an {\it isometric copy} of $G_k$.  Note that each edge of $G$, while not necessarily having length 1, does have constant length as it is isometric to $G_k$.  We will let $G_{n,k}$ denote the unique largest isometric copy of $G_k$ in $G_n$.  We will call two points in $G_{k,1} \subset G_k$ that have edge degree 3 {\it fork points}.

Note that each Laakso graph has only two vertices with edge degree one, which we will denote the {\it terminals}.  We will choose one arbitrarily to call the {\it source} $s$ and the other the {\it sink} $t$.  This imposes a direction on each edge and a partial ordering on the graph (and all subgraphs).  We will choose a partial ordering so that geodesics going from source to sink are increasing.  Given a Laakso subgraph $G$ of $G_n$, we let $s(G)$ and $t(G)$ denote the source and sink of the subgraph $G$, chosen so that the induced partial subordering agrees with the partial ordering from $G_n$.  We say two points in $G_n$ are in series if there exists a geodesic from $s(G_n)$ to $t(G_n)$ so that passes through both points.  Otherwise, we say those two points are in parallel.

For each decreasing sequence of positive numbers $\{\theta_j\}_{j=1}^\infty$, we can define an embedding of the Laakso graphs $f : G_n \to \H$ by first defining the image of $\pi \circ f$ in $\R^2$.  On the subgraph $G_{n,1}$, the source-sink geodesic on the top of $G_{n,1}$ gets mapped to one particular piecewise linear curve from $\pi(f(s(G_n)))$ to $\pi(f(t(G_n)))$ as shown where all the line segments in the piecewise linear curve are of the same length to be determined.  Likewise, the source-sink geodesic on the bottom will get mapped to the other, again all line segments will be of the same length.  We will let $st(G_n)$ denote the line segment in $\R^2$ connecting $\pi(f(s(G_n)))$ to $\pi(f(t(G_n)))$.  We specify that the angle the diamonds make with $st(G_n)$ is $\frac{\theta_1}{2}$.  The picture looks like a double diamond with two line segments jutting out.  See Figure \ref{g:double-diamond}.  Here, each of the marked angles of the diamonds on the right have angle $\theta$ (the same then naturally holds for the angle on the opposite end of the diamonds).  Thus, $st(G_n)$, the line going through the diamonds horizontally, bisects these angles.

Now suppose we have defined how $\pi \circ f$ acts on $G_{k,1}$ in an unscaled copy of $G_k \subseteq G_n$.  Each edge of $G_{k,1}$ are the terminals of an unscaled copy of $G_{k-1}$ in $G_k$.  We then define $\pi \circ f$ on $G_{k-1,1}$ using the double diamond embedding of Figure \ref{g:double-diamond} except using $st(G_{k-1})$ as the axis.  We will also specify that the angle the diamonds make with the axis will be $\frac{1}{2} \theta_{n-k+2}$.  Continuing this construction, we get that $\pi \circ f$ is eventually defined for all of $G_n$.  Then if we specify that $f$ maps each edge of $G_n$ to a horizontal line segment of length 1, we get that $\pi \circ f$ uniquely determines how $f$ embeds all of $G_n$ up to translation and rotation.

\begin{figure}
\centering
  \scalebox{0.7}{
    \includegraphics{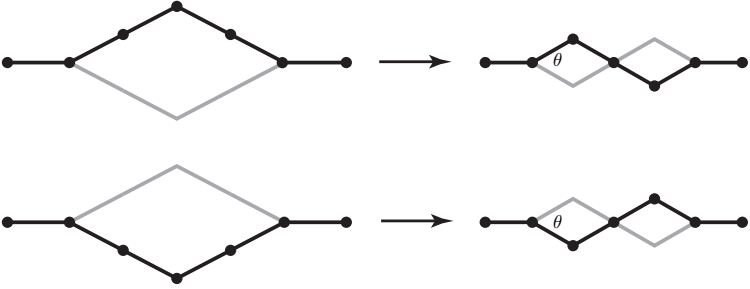}
    \setlength{\unitlength}{1mm}
    \begin{picture}(0,0)
      \put(-39,41){$\theta$}
      \put(-39,10.5){$\theta$}
    \end{picture}
  }
  \caption{The double diamond embedding \label{g:double-diamond}}
\end{figure}

Given two points $a,b \in \R^2$, we will let $\overline{ab}$ denote the line segment connecting $a$ to $b$.  We will let $\tilde{G}_n$ denote $\pi(f(G_n))$, the projection of the image of $G_n$ to $\R^2$.  Given a point $x \in G_n$, we will use $\tilde{x}$ as shorthand to denote $\pi(f(x))$.  Thus, $\tilde{s}(G),\tilde{t}(G)$ denote the points in $\tilde{G}_n$ corresponding to the terminal points $s(G),t(G)$.  Given Laakso subgraphs $G$ and $G'$, we will say that $\theta$ is the angle between them if the angles defined by $st(G)$ and $st(G')$ differ by $\theta$.  It follows easily from the construction of $f$ that if $G$ and $G'$ are Laakso subgraphs in series that share a terminal point, then the angle between them is at most $\theta_1$.

\begin{lemma} \label{short-cut}
  Let $P$ be a source-sink geodesic path in $G_n$.  The closed path in $\R^2$ that goes from $\tilde{s}(G_n)$ to $\tilde{t}(G_n)$ via $\pi \circ f(P)$ and then goes straight back to $\tilde{s}(G_n)$ encloses a region of zero signed area.
\end{lemma}

\begin{remark}
  The point of this lemma is that if we have a closed path in $\R^2$ that travels along the image of a source-sink geodesic of some isometric copy $G_k \subset G_n$ under $\pi \circ f$, then we can replace this portion with just $st(G_k)$ without changing the signed area.
\end{remark}

\begin{proof}
  The lemma is trivial for $G_1$.  Now suppose we have proven the lemma for $G_k$ up to $k = n-1$ and let $f : G_n \to \H$ be the double diamonds mapping.  Then any source-sink geodesic can be thought of as the concatenation of six source-sink geodesics through unscaled copies of $G_{n-1}$, each one with terminals in $G_{n,1}$.  Thus, we can use the inductive assumption and the previous remark to get that the signed area of the closed path needed is the same as the signed area of the the corresponding path via straight lines through points of $G_{n,1}$.  Then the statement holds again by the case of $G_1$.
\end{proof}

\begin{lemma} \label{st-distortion}
  Let $\{\theta_j\}_{j=1}^\infty$ be any decreasing sequence of positive numbers and let
  \begin{align*}
    L_{\ell,m} := \frac{6^n}{\prod_{j=\ell}^m (2+4\cos\theta_j)}.
  \end{align*}
  If $f : G_n \to \H$ is a double diamond embedding with angles $\{\theta_j\}$, then for any unscaled copy of $G_k$ in $G_n$, we have that
  \begin{align*}
    d(f(s(G_k)),f(t(G_k))) = \frac{6^k}{L_{n-k+1,n}}.
  \end{align*}
\end{lemma}

\begin{proof}
  Lemma \ref{short-cut} gives that $NH(f(s(G_k))^{-1}f(t(G_k))) = 0$.  When $NH(x^{-1}y) = 0$, we have that $d(f(x),f(y)) = |\tilde{x} - \tilde{y}|$ so it suffices then to show that
  \begin{align*}
    |st(G_k)| = \frac{6^k}{L_{n-k+1,n}}.
  \end{align*}
  The case when $k = 1$ is straightforward trigonometry in $\R^2$.  Suppose we have shown the statement for up to $k-1$ and consider an unscaled copy of $\tilde{G}_k$ in $\tilde{G}_n$.  Then $\tilde{G}_k$ can be expressed as ten unscaled copies of $\tilde{G}_{k-1}$ glued together via using the edge structure of $\tilde{G}_{k,1}$.  By induction, $|st(G_{k-1})| = 6^{k-1} L_{n-k+2,n}^{-1}$.  Thus, we can use the $k=1$ case except the angle is now $\theta_{n-k+1}$ and the edge lengths are $|st(G_{k-1})|$ to get that
  \begin{align*}
    |st(G_k)| = \frac{6|st(G_{k-1})|}{2 + 4 \cos\theta_{n-k+1}} = \frac{6^k}{L_{n-k+1,n}}.
  \end{align*}
\end{proof}

Let $x \in G_{n,1}$ and $y \in G_n$.  We will write out a specific path from $x$ to $y$.  Let $p_1 = x$.  If we choose $k$ so that $6^k \leq d_{G_n}(x,y) < 6^{k+1}$ then there exists some $p_2 \in G_n$ so that $d(p_2,x) = 6^k$ and
\begin{align*}
  d(x,y) = d(x,p_2) + d(p_2,y).
\end{align*}
If there are multiple choices for $p_2$, choose one arbitrarily.  Another way to say this is that $p_2$ and $x$ form the terminals of the largest unscaled copy of $G_k$ on the geodesic path from $x$ to $y$.  Now suppose $p_j$ has been defined.  If $p_j = y$, then we stop.  Otherwise, let $k$ be so that $6^k \leq d(p_j,y) < 6^{k+1}$ and choose a $p_{j+1}$ (breaking ties arbitrarily if there are multiple options) so that $d(p_j,p_{j+1}) = 6^k$ and
\begin{align*}
  d(p_j,y) = d(p_j,p_{j+1}) + d(p_{j+1},y).
\end{align*}
Note that this process will eventually stop, giving us a sequence of points $\{p_i\}_{i=1}^N$ connecting $x$ to $y$.  We will call this path $\{p_i\}_{i=1}^N$ a {\it developed path} from $x$ to $y$.  There is no guarantee of uniqueness for developed paths.  This can be thought of as a kind of base 6 numbering system of terminal geodesics.  Notice that if $p_i$ is a point in the developed path from $p_1$ to $p_N$, then a valid developed path of $p_1$ to $p_i$ is exactly $\{p_1,...,p_i\}$ and a valid developed path from $p_i$ to $p_N$ is exactly $\{p_i,...,p_N\}$.

If we let $a_i = \log_6 d(p_i,p_{i+1})$, we would get that $p_i$ and $p_{i+1}$ are terminal points for unscaled copies of $G_{a_i}$ on the geodesic path from $x$ to $y$.  Notice that $a_i$ is a nonincreasing sequence of numbers and each distinct number in $\{a_i\}_{i=1}^{N-1}$ can only appear at most five times.  We then let $\lambda(y;x)$ denote the number of distinct numbers in $\{a_i\}$.  Thus, if $\lambda(y;x)$ is small, then the developed path from $x$ to $y$ does not have to change scales many times (although the changes in scales it makes can be large).  We have that the angle that the line segment $\overline{\tilde{p}_i\tilde{p}_{i+1}}$ makes with $st(G_n)$ is at most
\begin{align}
  \sum_{j=1}^{\lambda(p_{i+1};p_1)} \theta_j \leq \lambda(p_{i+1};p_1) \theta_1. \label{developed-angle}
\end{align}
It is also easy to see that
\begin{align*}
  d(x,y) = \sum_{i=1}^{N-1} d(p_i,p_{i+1}) = \sum_{i=1}^{N-1} 6^{a_i} \leq 5 \frac{6^{a_1}}{1 - \frac{1}{6}} \leq 6^{a_1 + 1}.
\end{align*}

\begin{lemma} \label{convex-hull}
  $\tilde{G}_n$ is contained in the closed convex hull $C$ of $\tilde{G}_{n,1}$.
\end{lemma}

\begin{remark} \label{aperature}
  If the diamonds of $\tilde{G}_{k,1}$ have an angle of $\theta$, then the aperatures of the convex hull of $\tilde{G}_{k,1}$ at the terminals of $\tilde{G}_{k,1}$ have angles of
  \begin{align*}
    2\tan^{-1} \left( \frac{\sin \frac{\theta}{2}}{1 + \cos \frac{\theta}{2}} \right) = \frac{\theta}{2}.
  \end{align*}
\end{remark}

\begin{proof}
  We will claim by induction that all unscaled copies of $\tilde{G}_k$ in $\tilde{G}_n$ are contained in the convex hulls of $\tilde{G}_{k,1}$.  It is straightforward to see that all unscaled copies of $\tilde{G}_1$ in $\tilde{G}_n$ are contained in their closed convex hulls.  Now suppose we have proven this for up to $k-1$.  Consider an unscaled copy of $\tilde{G}_k$.  It is composed of 10 copies of unscaled copies of $\tilde{G}_{k-1}$, each of which is contained in the convex hulls of their respective $\tilde{G}_{k-1,1}$.  We have that the convex hulls of $\tilde{G}_{k-1,1}$ make an angle $\frac{1}{2}\theta_{n+2-k}$ at the terminals.  As the convex hull of $\tilde{G}_{k,1}$ makes an angle of $\frac{1}{2}\theta_{n+1-k} > \frac{1}{2}\theta_{n+2-k}$ at the terminals of $\tilde{G}_k$, we get by elementary planar geometry that each of the convex hulls of $\tilde{G}_{k-1,1}$ is contained in the convex hull of $\tilde{G}_{k,1}$.  This finishes the proof.
\end{proof}

Given Laakso subgraphs $G$ and $G'$ of $G_n$, we can define the angle the convex hulls of $\tilde{G}$ and $\tilde{G'}$ make as just the angles $G$ and $G'$ make.

\begin{lemma} \label{fork-collapse}
  Let $z$ be a fork point in $G_{n,1}$.  If $x,y \in G_n$ are parallel points in different unscaled copies of $G_{n-1}$ such that $d(z,x) = d(z,y)$, then
  \begin{align*}
    |\tilde{x} - \tilde{y}| \leq 12 \theta_1 d(z,x).
  \end{align*}
\end{lemma}

\begin{proof}
  We may suppose without loss of generality that $z$ is closer to the source than the sink.  As $x$ and $y$ are parallel points in different unscaled copies of $G_{n-1}$ with $d(x,z) = d(y,z)$, we get that one of $x$ and $y$ must be on a $G_{n-1}$ making up the ``top'' of a diamond while the other is on the $G_{n-1}$ making up the ``bottom''.  Note that $\tilde{G}_n$ is invariant when reflected about $st(G_n)$.  We let $x_1 = x$ and we let $x_2$ denote the point in $G_n$ such that $\tilde{x}_2$ is the reflection of $\tilde{x}_1$ about $st(G_n)$.  Then it follows that
  \begin{align*}
    |\tilde{x}_2 - \tilde{z}| = |\tilde{x}_1 - \tilde{z}| \leq d(z,x),
  \end{align*}
  and $\overline{\tilde{z}\tilde{x}_2}$ makes an angle of no more than $\frac{1}{2} (\theta_1 + \theta_2) \leq \theta_1$ with $st(G_n)$.  The second claim follows easily from Remark \ref{aperature} and the fact that $x_2$ is contained in the convex hull $C$ of some $\tilde{G}_{n-1}$ for some unscaled copy of $G_{n-1}$ in $G_n$.  The same follows for $\overline{\tilde{z}\tilde{x}_1}$ and so we get that
  \begin{align*}
    |\tilde{x}_1 - \tilde{x}_2| \leq 2 \theta_1 |\tilde{z} - \tilde{x}| = 2\theta_1 d(z,x).
  \end{align*}
  If $x_2 = y$, we stop.  Otherwise, we get that $x_2$ and $y$ are contained in the same unscaled copy of $G_{n-1}$ and satisfy $d(z,x_2) = d(z,y)$.  Thus, $x_2$ and $y$ are contained in different unscaled copies of $G_{k-1}$ of some unscaled copy of $G_k$ where $k < n$ both copies (of $G_{k-1}$) of which are contained in the unscaled copy of $G_{n-1}$ containing $x_2$ and $y$.  Thus, as before, there exists some fork point $z_2$ in $G_{k,1}$ so that $d(z_2,x_2) = d(z_2,y) \leq \frac{5}{6} d(z,x)$.  If we let $x_3$ denote the point in $G_k$ such that $\tilde{x}_3$ is the reflection of $\tilde{x}_2$ about the axis of $st(G_k)$, we get as before that
  \begin{align*}
    |\tilde{x}_3 - \tilde{x}_2| \leq 2 \theta_1 d(z_2,x_2) \leq \frac{5}{6} 2\theta_1 d(z,x).
  \end{align*}
  Continuing, we get a sequence of points $\{x_1,...,x_N\}$ where $x_1 = x$ and $x_N = y$ so that
  \begin{align*}
    |\tilde{x}_1 - \tilde{x}_N| \leq \sum_{i=1}^{N-1} |\tilde{x}_i - \tilde{x}_{i+1}| \leq 2 \theta_1 d(z,x) \sum_{i=0}^{N-1} \left( \frac{5}{6} \right)^k \leq 12 \theta_1 d(z,x).
  \end{align*}
\end{proof}

Let
\begin{align*}
  L := \lim_{n \to \infty} \frac{6^n}{\prod_{j=1}^n (2+4\cos\theta_j)}.
\end{align*}
If $\theta_j$ is of the form
\begin{align}
  \theta_j = \left( \sqrt{M+j} \log(M+j) \right)^{-1} \label{e:thetaj-defn}
\end{align}
for $M \geq 1$, then we see that $L$ is bounded.  For convenience, we let $\log$ be in base 2 unless specified otherwise.  In fact, as $M$ increases, $L$ decreases and so we may suppose that $L$ is always bounded by some absolute constant.  It is not hard to see that, by specifying $M$ larger than some absolute constant, we may suppose that $L \leq 2$ always, which we will do from now on.  We will still continue using $L$, just to avoid magic numbers.

\begin{proposition} \label{first-level-distortion}
  There exist absolute constants $C > 0$ and $M_0 \geq 1$ so that if $x,y \in G_n$ are in different unscaled copies of $G_{n-1}$ and $\theta_j$ is of the form \eqref{e:thetaj-defn} for any $M \geq M_0$, then
  \begin{align*}
    \frac{C}{L}\theta_1^{1/2} d(x,y) \leq d(f(x),f(y)) \leq d(x,y).
  \end{align*}
\end{proposition}

\begin{proof}
  Remember that we will always be free to choose $M_0$ large enough so that $L \leq 2$.  The upper 1-Lipschitz inequality is trivial as $f$ is 1-Lipschitz on each edge of $G_n$ and the metric on $G_n$ is a path metric.  Thus, we only need to prove the lower bound.  We let $P$ denote a geodesic path from $x$ to $y$.  There are two cases.  Either $x$ and $y$ are in series or they are in parallel.
  
  {\bf Case 1:} $x$ and $y$ are in series.

  We will actually show the stronger statement that if $x$ and $y$ are in series, then
  \begin{align}
    |\tilde{x} - \tilde{y}| \geq \frac{1}{12L} d(x,y). \label{series-bound}
  \end{align}
  This clearly gives the required result as $d(f(x),f(y)) \geq |\tilde{x} - \tilde{y}|$.


  Because $x$ and $y$ are in different unscaled copies of $G_{n-1}$, we have that there exists some element $z \in (P \backslash \{x,y\}) \cap G_{n,1} \neq \emptyset$.  We may suppose without loss of generality that $d(x,z) \geq d(y,z)$ and let $k \in \N$ be such that
  \begin{align}
    6^k \leq d(x,z) < 6^{k+1}. \label{log-xz-series}
  \end{align}
  Let $\{G^{(i)}\}_i$ denote all the unscaled copies of $G_k$ between $x$ and $z$.  There is at least 1 and at most 5 of them because of \eqref{log-xz-series}.  We also let $z'$ denote the terminal of $G^{(i)}$ that is closest to $x$.  If we let $G'$ be the unscaled copy of $G_k$ containing $x$ with terminal $z'$ and $G$ be the unscaled copy of $G_k$ containing $y$ with terminal $z$, then $G^{(i)}$ connect $G$ and $G'$.
  
  Note that the angle between each $st(G^{(i)})$ is at most $2\theta_1$ and that $|st(G^{(i)})| \geq \frac{6^k}{L}$.  Let $Q : \R^2 \to \R$ denote the orthogonal projection in $\R^2$ to the line spanned by $\tilde{z}$ and $\tilde{z'}$.  As $\tilde{z'}$ and $\tilde{z}$ are opposite terminals of a chain of at most five $\tilde{G}_k$ each of which make angles at most $2\theta_1$ with its neighbor, if we let $\theta_1$ be small enough, then we get that the linear ordering (or its opposite) is preserved
  \begin{align*}
    Q(\tilde{x}) \leq Q(\tilde{z'}) \leq Q(\tilde{z}) \leq Q(\tilde{y})
  \end{align*}
  and
  \begin{align*}
    |Q(\tilde{z'}) - Q(\tilde{z})| \geq \frac{6^k}{L}.
  \end{align*}
  Thus, we get that
  \begin{align*}
    |\tilde{x} - \tilde{y}| \geq |\tilde{z'} - \tilde{z}| \geq \frac{6^k}{L} \overset{\eqref{log-xz-series}}{>} \frac{1}{6L} d(x,z) \geq \frac{1}{12L} d(x,y).
  \end{align*}

  {\bf Case 2:} $x$ and $y$ are in parallel.
  
  Let $z$ be a fork point in $G_{n,1}$ that lies in a geodesic path from $x$ to $y$.  Then $d(x,y) = d(x,z) + d(z,y)$.  We will assume without loss of generality that $d(x,z) \geq d(y,z)$ and $x > z$, $y > z$.  We can suppose that $d(y,z) \leq \frac{1}{3} 6^n$ as otherwise we have that $\min \{d(x,z),d(y,z)\} > \frac{1}{3}6^n$ and so $z$ could not be on the geodesic path between $x$ and $y$.

  Let $k \in \N$ satisfy
  \begin{align}
    6^k \leq d(y,z) < 6^{k+1}, \label{log-yz-parallel}
  \end{align}
  and $\alpha \in \N$ be the smallest integer such that
  \begin{align}
    6^\alpha > 6^8 L^2. \label{define-alpha}
  \end{align}
  As $L \leq 2$, one sees that $\alpha = 9$, although we will still continue using $\alpha$ for clarity.

  Let $\ell$ and $m$ be the number of unscaled copies of $G_{k-\alpha}$ on the geodesic paths from $z$ to $x$ and $y$, respectively.  It is easy to see from \eqref{log-yz-parallel} that
  \begin{align}
    6^\alpha \leq m < 6^{\alpha+1}. \label{e:m-const}
  \end{align}

  {\bf Case 2a:} $\ell \leq m + 1$.  Thus, $\ell \leq 6^{\alpha+1}$ and $d(x,y) \leq 6^{k+2}$.

  Note that the mapping $f$ takes a geodesic path from $x$ to $y$ traveling through $z$ to a horizontal path $P$ in $\H$ from $f(x)$ to $f(y)$ going through $f(z)$.  Consider the developed path $\{p_i\}_{i=1}^N$ from $z$ to $x$.  If there exists some $i \in \{1,...,N-1\}$ so that the angle of the oriented line segment from $\overline{\tilde{p}_i\tilde{p}_{i+1}}$ makes an angle of more than $\frac{\pi}{4}$ with $st(G_n)$, then we let $a = p_i$ for the minimal such index $i$.  Otherwise, we let $a = x$.  We do the same thing with $y$ to get a point $b$.

  We first claim that
  \begin{align}
    \max\{d(a,x), d(b,y)\} \leq \frac{\theta_1^{1/2}}{50L} 6^k. \label{tail-bounds}
  \end{align}
  We will just prove the inequality for $d(a,x)$ as the inequality for $d(b,y)$ will follow from the same reasoning.  If $a = x$, then the statement is obviously true.  Thus, we may suppose $a \neq x$ and so the angle between $\overline{\tilde{p}_i\tilde{p}_{i+1}}$ and $st(G_n)$ is greater than $\frac{\pi}{4}$.  We then have that
  \begin{multline*}
    \frac{\pi}{4} \leq \sum_{j=1}^{\lambda(a;z)} \theta_j + \theta_{\lambda(a;z)} \leq 2 \sum_{j=1}^{\lambda(a;z)} \frac{1}{\sqrt{M + j}} \leq 2 \int_0^{\lambda(a;z)} \frac{dx}{\sqrt{M + x}} \\
    = 4 \left( \sqrt{M + \lambda(a;z)} - \sqrt{M} \right).
  \end{multline*}
  Using concavity estimates for the square root at $M$, we get that this implies that $\lambda(a;z) \geq \frac{\pi}{8} \sqrt{M}$.  We then have that
  \begin{align*}
    d(a,x) = \sum_{j=i+1}^{N-1} d(p_j,p_{j+1}) \leq 6^k \cdot 5\sum_{j=\lambda(a;z)}^\infty 6^{-j} \leq 6^{k-\lambda(a;z)+1} \leq 6^{k+1 - \frac{\pi}{8} \sqrt{M}}.
  \end{align*}
  As $\theta_1 = (1 + M)^{-1/2}$, if we choose $M$ to be larger than some absolute constant, we have that
  \begin{align*}
    d(a,x) \leq \frac{\theta_1^{1/2}}{50L} 6^k,
  \end{align*}
  which finishes the proof of the claim.  Here we used a very inefficient bound that $6^{-c\theta^{-1}}$ is much less than $\theta^{1/2}$ for small enough $\theta$.

  We let $\Sigma$ denote the signed area of the closed path in $\R^2$ that first goes from $\tilde{a}$ to $\tilde{b}$ along the image of $P$ via $\pi \circ f$ and then goes back to $\tilde{a}$ via a straight line.  By \eqref{swept-area},
  \begin{align*}
    d(f(a),f(b)) \geq NH(f(a)^{-1}f(b)) = |\Sigma|^{1/2}.
  \end{align*}
  We will break up $\Sigma$ into the sum of the signed area of two separate closed paths.

  Let $u,v \in G_n$ be points on the geodesic path from $z$ to $x$ and $y$ so that $d(u,z) = d(v,z) = m 6^{k-\alpha}$, that is $v$ is the point in the unscaled copies of $G_{k-\alpha}$ that is closest to $y$.  If we set $M$ large enough, then we can ensure that each of the $m$ possible $st(G_{k-\alpha})$ between $z$ and $u$ and $v$ makes an angle of no more than $\pi/4$ with $st(G_n)$ and so $z < u \leq a$ and $z < v \leq b$.  This can be done simply because \eqref{e:m-const} says that $m$ is comparable to $6^\alpha$, a constant value.  Thus, as successive $st(G_{k-\alpha})$ differ by angles of no more than $\theta_1$, we can easily bound the total accrued angle deviation.
  
  Note that $v$ is necessarily in the developed path from $z$ to $b$ as $d(v,y) < 6^{k-\alpha}$.  We then let $\{p_1,...,p_N\}$ denote the developed path from $v$ to $b$.  There are two cases for $u$.  The first case is that it is in the developed path from $z$ to $a$, in which case we let $\{p_1',...,p_{N'}'\}$ denote the developed path from $u$ to $a$.  Otherwise, $u$ is not in the developed path from $z$ to $a$.  In this case, it must be because $6^{k-\alpha} \leq d(u,a) < 6^{k-\alpha+1}$, and so there exists a unique point $u'$ for which $u < u' \leq a$, $d(u,u') < 6^{k-\alpha}$, and $u'$ is in the developed path from $z$ to $a$.  We then let $\{p_1',...,p_{N'}'\}$ denote the path so that $p_1' = u$, $p_2' = u'$, and then $p_i'$ follows the developed path from $u'$ to $a$.  We then let $T$ denote the signed area of the closed path that goes from $\tilde{p}_{N'}' = \tilde{a}$ via the line segments $\overline{\tilde{p}_{i+1}'\tilde{p}_i'}$ until it reaches $\tilde{p}_1' = \tilde{u}$ where it proceeds to go straight to $\tilde{p}_1 = \tilde{v}$ and then goes via the line segments $\overline{\tilde{p}_i \tilde{p}_{i+1}}$ until it reaches $\tilde{p}_N = \tilde{b}$ and then it goes straight back to $\tilde{a}$.
  
  Let $z = q_1 < q_2 < ... < q_m = u$ and $z = q_1' < q_2' < ... < q_m' = v$ be the points so that $d(q_i,q_{i+1}) = 6^{k-\alpha}$ and $d(q_i',q_{i+1}') = 6^{k-\alpha}$, that is $q_i$ and $q_i'$ are terminal points of successive unscaled copies of $G_{k-\alpha}$ from $z$ to $u$ and $v$, respectively.  We let $\Sigma'$ denote the signed area of the path in $\R^2$ that goes from $\tilde{u}$ via the line segments $\overline{\tilde{q}_{i+1}\tilde{q}_i}$ until it reaches $\tilde{q}_1 = \tilde{z} = \tilde{q}_1'$ where it proceeds to then go via line segments $\overline{\tilde{q}_i'\tilde{q}_{i+1}'}$ until it reaches $\tilde{v}$ and then it goes straight back to $\tilde{u}$.  Then Lemma \ref{short-cut} tells us that
  \begin{align*}
    \Sigma = \Sigma' + T.
  \end{align*}

  \begin{figure}
  \centering
    \scalebox{0.7}{\includegraphics{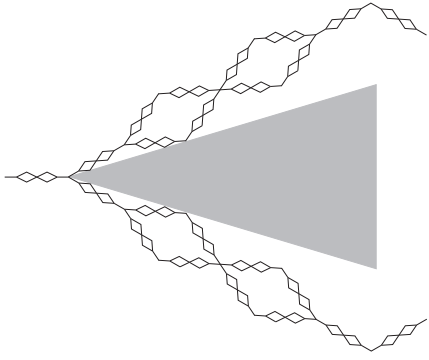}}
    \caption{The triangle that fits between two copies of $\tilde{G}_k$ \label{g:sigma-prime}}
  \end{figure}
  
  We first prove that the path defining $\Sigma'$ does not self-intersect.  Recall that $d(y,z) \leq \frac{1}{3} 6^n$.  Thus, the line segments defined by $\{q_1,...,q_m\}$ and $\{q_1',...,q_m'\}$ do not intersect as they lie in different unscaled copies of $G_{n-1}$ that are parallel.  As $m < 6^{\alpha+1}$, all the line segments $\overline{\tilde{q}_i\tilde{q}_{i+1}}$ makes an angle of at most $6^{\alpha + 1} \theta_1$ with $st(G_n)$.  Here, we again use the fact the inefficient---but sufficient---fact that any successive $st(G_{k-\alpha})$ cannot have angles that differ by more than $\theta_1$.  The same holds for $\overline{\tilde{q}_{i+1}'\tilde{q}_i'}$.  Note that the piecewise linear curves connecting the $\{\tilde{q}_i\}$ and the $\{\tilde{q}_i'\}$ lie in convex hulls of unscaled copies of $G_{n-1}$ that on opposite sides of $st(G_n)$ and so they are are disjoint.  Thus, by specifying that $\theta_1$ be small enough we get from Lemma \ref{fork-collapse} that $\overline{\tilde{u}\tilde{v}}$ is small enough (for example, smaller than the length of $\overline{\tilde{q}_i\tilde{q}_{i+1}}$).  Now one can view $st(G_n)$ as the $x$-axis of $\R^2$ and the curves defined by $\overline{\tilde{q}_i\tilde{q}_{i+1}}$ and $\overline{\tilde{q}_i'\tilde{q}_{i+1}'}$ as piecewise linear graphs starting from the origin.  These curves are on opposite sides of the $x$-axis and have controlled angle deviations so that they do not self-intersect but the endpoints wind up being much closer to each other compared to the size of the line segments from which non-intersection follows.
  
  As $m \geq 6^{\alpha}$, we also have that each subpath from $\tilde{z}$ to $\tilde{x}$ and $\tilde{y}$ in the projection of $P$ to $\R^2$ via $\pi \circ f$ contains the projection to $\R^2$ of source-sink geodesics in unscaled copies $G$ and $G'$ of $G_k$.  We get from the fact that $\tilde{G}$ and $\tilde{G'}$ are contained within the convex hulls of $\tilde{G}_1$ and $\tilde{G'}_1$ and Lemma \ref{st-distortion} that there is an isosceles triangle of angle $\theta_1 - \frac{\theta_2}{2}$ with side lengths at least
  \begin{align*}
    \frac{6^k}{L} \cos \frac{\theta_1}{2} \left[ \cos \left( \frac{2\theta_1 - \theta_2}{4} \right) \right]^{-1}
  \end{align*}
  that can fit between them in $\Sigma$.  See Figure \ref{g:sigma-prime}.  We then have
  \begin{align}
    |\Sigma'| \geq \left( \frac{6^k}{L} \cos \frac{\theta_1}{2} \right)^2 \tan \left( \frac{2\theta_1 - \theta_2}{4} \right) > \frac{1}{6} \left( \frac{6^k}{L} \right)^2 \theta_1. \label{sigma'-bound}
  \end{align}
  Here, we need to specify that $\theta_1$ is smaller than some absolute constant and use the fact that $\theta_2 < \theta_1$.

  We now bound $|T|$.  As there are only at most two line segments $\overline{\tilde{u}\tilde{v}}$ and $\overline{\tilde{a}\tilde{b}}$ of $\partial T$, the boundary curve defining $T$, that can make an angle of no more than $\frac{\pi}{4}$ with $st(G_n)$, we have that the winding number of $\partial T$ around any point is at most 1.  Thus, $|T|$ is no more than the unsigned area enclosed by $\partial T$ and so it suffices to bound the unsigned area.
  
  We have that $d(v,b) \leq d(v,y) < 6^{k-\alpha}$.  Thus, by Lemma \ref{convex-hull} the path from $\tilde{v}$ to $\tilde{b}$ along the projection of $P$ is contained in a convex hull $C_1$ of the image of some unscaled copy of $G_{k-\alpha}$ where $\tilde{v}$ is a terminal.  In the same way, we have $|\tilde{a} - \tilde{u}| \leq 2 \cdot 6^{k-\alpha}$ and so the path from $\tilde{a}$ to $\tilde{u}$ is contained in either one or two convex hulls $C_2$ and $C_3$ of one or two (sequential) unscaled copies of $G_{k-\alpha}$ such that $\tilde{u}$ is the terminal of one of them.  Thus, the region in $\R^2$ given by $T$ is contained in the convex hull of $C_1,C_2,C_3$.

  We will collect some information about the relative geometry of the $C_i$.  As $d(z,u) = d(z,v) = m6^{k-\alpha} < 6^{k+1}$, we have by Lemma \ref{fork-collapse} that $|\tilde{u} - \tilde{v}| \leq 6^{k+3} \theta_1$.  Also, we have that each $C_i$ makes an angle of at most $2(6^{\alpha+1}+2)\theta_1$ with $st(G_n)$.  This follows from the fact that $m < 6^{\alpha+1}$.  We also have that $C_2$ and $C_3$ can make an angle of at most $2\theta_1$ with each other.  Finally, the line segments connecting the endpoints of $C_i$ each have length no more than $6^{k-\alpha}$.

  Now consider the following domain $S$ in $\R^2$.  We start with a parallelogram which has $\overline{\tilde{u}\tilde{v}}$ as the left side and whose top and bottom are of length $6^{k-\alpha+2}$ and are parallel to $st(G_n)$.  We take the top and bottom of the parallelogram to be the bases of isosceles triangles whose vertex angle are both $\pi - 8(6^{\alpha+1}+2) \theta_1$ (thus they are congruent).  We then let $S$ be the region that is the parallelogram along with the two triangles.  See Figure \ref{g:s-t} for how $S$ is situated relative to $C_1$,$C_2$, and $C_3$ (the length ratios may not be accurate).

  \begin{figure}
    \centering
    \scalebox{0.6}{\includegraphics{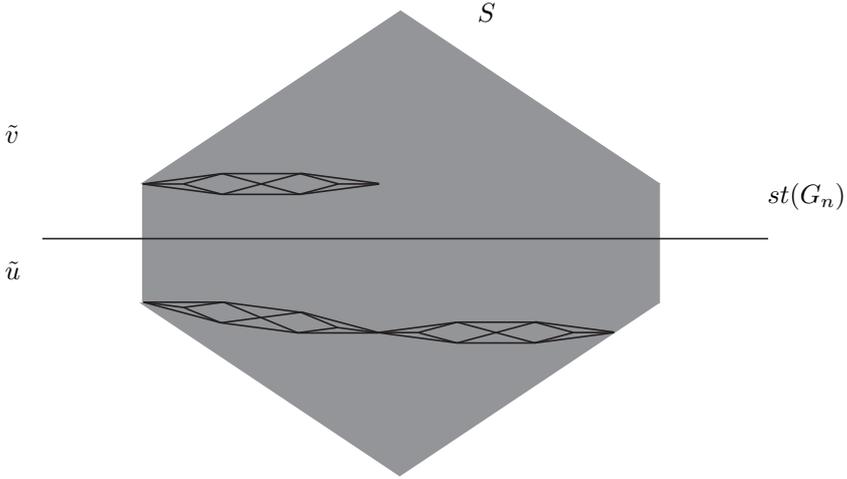}}
    \setlength{\unitlength}{1mm}
    \begin{picture}(0,0)
    \put(-112,44){ $\tilde{v}$}
    \put(-112,26){ $\tilde{u}$}
    \put(-12,36){ $st(G_n)$}
    \put(-50,60){ $S$}
    \end{picture}
    \caption{The geometry of $S$ relative to $C_1$,$C_2$, and $C_3$. \label{g:s-t}}
  \end{figure}
  
  Using all the information collected about the relative geometry of $C_1$,$C_2$, and $C_3$, it is elementary to see that the region given by $T$ is contained within $S$.  This gives us that
  \begin{multline}
    |T| \leq |S| \leq 6^{2k-\alpha+5} \theta_1 + \frac{1}{4} 6^{2k-2\alpha+4} \tan(4(6^{\alpha+1}+2) \theta_1) \\
    \leq 6^{2k-\alpha+5} \theta_1 + 6^{2k-2\alpha+5} (6^{\alpha+1}+2) \theta_1\overset{\eqref{define-alpha}}{\leq} \frac{1}{36} \left(\frac{6^k}{L} \right)^2 \theta_1. \label{t-bound}
  \end{multline}
  Thus, 
  \begin{multline*}
    d(f(x),f(y)) \geq d(f(a),f(b)) - d(f(a),f(x)) - d(f(b),f(y)) \\
    \overset{\eqref{tail-bounds}}\geq |\Sigma|^{1/2} - \frac{\theta_1^{1/2}}{4L} 6^k \geq (|\Sigma'| - |T|)^{1/2} - \frac{\theta_1^{1/2}}{50L} 6^k \overset{\eqref{sigma'-bound} \wedge \eqref{t-bound}}{\geq} \frac{\theta_1^{1/2}}{3L} 6^k,
  \end{multline*}
  and as $d(x,y) \leq 6^{k+2}$, we get
  \begin{align*}
    \frac{d(f(x),f(y))}{d(x,y)} \geq \frac{\theta_1^{1/2}}{108L}.
  \end{align*}

  {\bf Case 2b:} $\ell > m + 1$.

  Let $y_0$ denote a point in series with $x$ so that $y_0 > z$ and $d(y_0,z) = d(y,z)$.  Then
  \begin{align*}
    d(y_0,x) = d(x,z) - d(y_0,z) \geq \frac{1}{1 + 6^{1+\alpha}} d(x,z),
  \end{align*}
  and
  \begin{align*}
    |\tilde{y}_0 - \tilde{y}| \leq 12 \theta_1 d(y,z) \leq 12 \theta_1 d(x,z),
  \end{align*}
  by Lemma \ref{fork-collapse} and the fact that we've assumed $d(x,z) \geq d(y,z)$.  Thus, we get that
  \begin{align*}
    d(f(x),f(y)) &\geq |\tilde{x} - \tilde{y}| \\
    &\geq |\tilde{x} - \tilde{y}_0| - |\tilde{y}_0 - \tilde{y}| \\
    &\overset{\eqref{series-bound}}{\geq} \left( \frac{1}{12(1 + 6^{1+\alpha})L} - 72\theta_1 \right) d(x,z) \\
    &\geq \frac{1}{2} \left( \frac{1}{(1 + 6^{1+\alpha})12L} - 72\theta_1 \right) d(x,y)
  \end{align*}
  Remember that we may assume $\alpha$ and $L$ are both bounded by some absolute constant regardless of how small we assign $\theta_1$.  Thus, if we choose $\theta_1$ to be smaller than some absolute constant, we get that there exists some other absolute constant $C > 0$ so that
  \begin{align*}
    d(f(x),f(y)) \geq C d(x,y).
  \end{align*}
\end{proof}

\begin{proof}[Proof of Theorem \ref{laakso-distortion}]
  Let $C > 0$ and $M_0 \geq 1$ be as defined in Proposition \ref{first-level-distortion} and define
  $$\theta_j = \left( \sqrt{M_0 + j} \log(M_0 + j) \right)^{-1}.$$
  Let $x,y \in G_n$.  We will show that
  \begin{align*}
    \frac{C}{L(M_0 + n)^{1/4} \sqrt{\log(M_0 + n)}} d(x,y) \leq d(f(x),f(y)) \leq d(x,y),
  \end{align*}
  which clearly finishes the proof of the theorem.

  Consider the largest unscaled copy of some Laakso subgraph $G_k$ of $G_n$ for which $x,y$ are both in $G_k$.  Then $x$ and $y$ must be in different unscaled copies of $G_{k-1}$ in $G_k$.  Note that $f$ restricted to $G_k$ acts as the double diamond embedding but where the first angle is $\theta_k$.  Thus, Proposition \ref{first-level-distortion} gives that
  \begin{align*}
    \frac{C}{L} \theta_k^{1/2} d(x,y) \leq d(f(x),f(y)) \leq d(x,y).
  \end{align*}
  Note then that
  \begin{align*}
    \theta_k^{1/2} = \frac{1}{(M_0 + k)^{1/4} \sqrt{\log(M_0 + k)}} \geq \frac{1}{(M_0 + n)^{1/4} \sqrt{\log(M_0 + n)}}.
  \end{align*}
  This finishes the proof of the theorem.
\end{proof}

\begin{proof}[Proof of Corollary \ref{optimal-convex}]
  This proof will resemble that of Proposition 3.1 in \cite{MN}.  We define a random walk on $G_m$ as follows.  For $t \leq 0$, we define $Z_t = s(G_m)$.  Then assuming $Z_t$ has been defined from $-\infty$ to $t \in \{1,...,6^m-1\}$ we let $Z_{t+1}$ to be the one (or two) neighboring points of $Z_t$ for which $Z_{t+1} > Z_t$.  If there are two choices for $Z_{t+1}$, choose either randomly with probability 1/2.  Finally, we let $Z_t = t(G_m)$ (that is, the sink) for all $t \geq 6^m$.

  As $d(f(x),f(y)) = 1$ when $x,y$ are neighbors of $G_m$, we get that
  \begin{align}
    \sum_{t \in \Z} \E[d(f(Z_t),f(Z_{t-1})^p] = 6^m. \label{edge-increment}
  \end{align}

  Fix $k \in \{1,...,m\}$ and let $h = \left\lceil \frac{\log 2}{\log 6} k \right\rceil$.  We can view $G_m$ as being built from $A = G_{m-h}$ where each edge of $G_{m-h}$ is replaced with a copy of $G_h$.  We claim that for every $i \in \{0,...,6^{m-h-1}-1\}$, $Z_t$ at time $t = 6^{h+1}i+6^h$ is located at a point in $G_m$ which has two outgoing edges, each one corresponding to a distinct copy of $G_h$.  This is because we can view $A$ as being built from $B = G_{m-h-1}$ with each edge in $B$ replaced by a $G_1$.  Note that each $G_1$ has a vertex, the lone neighbor of $s(G_1)$, of out degree 2.  The claim then follows as each edge in the $G_1$ is replaced by a copy of $G_h$ to form $G_m$.

  Consider the times
  \begin{align*}
    T_k = \{0,...,6^m-1\} \bigcap \left( \bigcup_{i=1}^{6^{m-h-1}-1} [6^{h+1}i+6^h + 6^{h-1}, 6^{h+1}i+6^h + 2 \cdot 6^{h-1}] \right).
  \end{align*}
  By definition of $h$, we have that
  \begin{align*}
    6^{h-1} < 2^k \leq 6^h.
  \end{align*}
  Thus, we get that if $t \in T_k$ such that $t \in [(6i + 1)6^h + 6^{h-1},(6i+1) 6^h + 2 \cdot 6^{h-1}]$ for some $i \in \{1,...,6^m-1\}$, then $t - 2^k \in [(6i+1)6^h - 6^h, (6i+1)6^h)$.  Thus, the walks $\{Z_s\}_{s \in \Z}$ and $\{\tilde{Z}_s(t-2^k)\}_{s \in \Z}$ at time $t' = (6i+1)6^h$ will have already become independent of each other and so they will select to walk down the two different $G_h$ branches with probability $\frac{1}{2}$.  Thus, we get from Theorem \ref{laakso-distortion} that there exists some $C > 0$ so that
  \begin{align*}
    \frac{\E[d(f(Z_t),f(\tilde{Z}_t(t-2^k)))]^p}{2^{kp}} &\geq \frac{1}{2} \frac{(2 \cdot 6^{h-1})^p}{C^p 2^{kp} m^{p/4} (\log m)^{p/2}} \\
    &\geq \frac{1}{6^p C^p m^{p/4} (\log m)^{p/2}}.
  \end{align*}
  Thus, we get that
  \begin{align*}
    \sum_{k=1}^\infty \sum_{t \in \Z} &\frac{\E[d(f(Z_t),f(\tilde{Z}_t(t-2^k)))]^p}{2^{kp}} \\
    &\geq \sum_{k=1}^m \sum_{t \in T_k} \frac{\E[d(f(Z_t),f(\tilde{Z}_t(t-2^k)))]^p}{2^{kp}} \\
    &\geq \frac{1}{6^p C^p}\sum_{k=1}^m \frac{|T_k|}{m^{p/4} (\log m)^{p/2}} \\
    &\gtrsim \frac{1}{6^p C^p} \sum_{k=1}^m \frac{6^{h-1} \cdot 6^{m-h-1}}{m^{p/4} (\log m)^{p/2}} \geq \frac{1}{6^p C^p} 6^m m^{1-\frac{p}{4}} (\log m)^{-p/2}.
  \end{align*}
  Now suppose $p < 4$.  Comparing the above inequality with \eqref{edge-increment} and noting that $m^{1-\frac{p}{4}} (\log m)^{-p/2} \to \infty$ as $m \to \infty$, we see that there cannot exist any finite $K > 0$ so that
  \begin{align*}
    \sum_{k=1}^\infty \sum_{t \in \Z}\frac{\E[d(f(Z_t),f(\tilde{Z}_t(t-2^k)))]^p}{2^{kp}} \leq K^p \sum_{t \in \Z} \E[d(f(Z_t),f(Z_{t-1})^p].
  \end{align*}
\end{proof}

\begin{proof}[Proof of Corollary \ref{discrete-nonembedding}]
  We will retain all the same notation as the proof of the previous corollary.  Let $g : \H \to \H(\Z)$ be the quasi-isometry with bounds \eqref{quasi-isometry}.  We will consider an embedding $f : G_n \to \H$ so that
  \begin{align}
    c_0(1+c_1) d(x,y) \leq d_\H(f(x),f(y)) \leq C n^{1/4} (\log n)^{1/2} d(x,y) \label{laakso-embedding}
  \end{align}
  for some absolute constant $C > 0$.  This is possible as $\H$ has a scaling automorphism.  As diam$(G_n) = 6^n$, we see that $g \circ f : G_n \to \H(\Z)$ maps $G_n$ into a ball of radius $2c_0 C n^{1/4} (\log n)^{1/2} 6^n$ when $n$ is sufficiently large (compared to $c_1$).  We also see that
  \begin{align*}
    d_{\H(\Z)}(g(f(x)),g(f(y))) \overset{\eqref{quasi-isometry} \wedge \eqref{laakso-embedding}}{\geq} (1 + c_1) d(x,y) - c_1 \geq d(x,y).
  \end{align*}

  Let $F : B(2c_0C6^n n^{1/4} (\log n)^{1/2}) \to X$ be a noncontracting map with Lipschitz constant $D$.  Then we get for large enough $n$ that $h = F \circ g \circ f : G_n \to X$ has the following bounds
  \begin{multline}
    d(x,y) \leq d_X(h(x),h(y)) \leq D \left( c_0 Cn^{1/4} (\log n)^{1/2} d(x,y) + c_1\right) \\
    \leq 2c_0 CD n^{1/4} (\log n)^{1/2} d(x,y). \label{discrete-markov}
  \end{multline}

  Let $\{X_t\}_{t \in \Z}$ be the same random walk on $G_n$ as in the proof of the previous corollary.  Then using the same reasoning from before, we have
  \begin{align}
    &\sum_{t \in \Z} \E[d(h(X_t),h(X_{t-1}))^p] \overset{\eqref{discrete-markov}}{\leq} (2c_0 CDn^{1/4} (\log n)^{1/2})^p 6^n, \label{increment} \\
    &\sum_{k=0}^\infty \sum_{t \in \Z} \frac{\E\left[ d(h(X_t),h(\tilde{X}_t(t-2^k)))^p \right]}{2^{kp}} \overset{\eqref{discrete-markov}}{\gtrsim} n 6^n. \label{drift}
  \end{align}
  As $X$ is Markov $p$-convex, we can use \eqref{markov-convexity-defn} to derive a lower bound for $D$ to get
  \begin{align*}
    c_X(B(7^n)) \geq c_X\left(B(2c_0 C6^n n^{1/4} (\log n)^{1/2})\right) \overset{\eqref{markov-convexity-defn} \wedge \eqref{increment} \wedge \eqref{drift}}{\gtrsim} \frac{n^{\frac{1}{p} - \frac{1}{4}}}{(\log n)^{1/2}}.
  \end{align*}
  This easily implies the lower bound \eqref{quantitative-nonembed} that we need.
\end{proof}

\section{Lower bounds for distortion of trees}

For ease of notation, we will write $\H_\infty$ more succintly in this section as $(\ell_2 \times \R, \cdot)$ where $\ell_2$ is now the $\ell_2$-sequence space of complex numbers.  The group product is then
\begin{align*}
  (x,t) \cdot (y,s) = \left(x+y, t + s + \frac{1}{2} \omega(x,y) \right)
\end{align*}
where $\omega(z,z') = \sum_{i=1}^\infty \Im(\overline{z_i}z_i')$.  As is well known, we can also express the symplectic form as $\omega(z,z') = \langle iz,z' \rangle$ where $\langle \cdot, \cdot \rangle$ is the usual inner product on $\ell_2$.  In this section, we prove that the complete binary trees $\{B_m\}_{m=1}^\infty$ embed into $\H_\infty$ with distortion at least $C\sqrt{\log \log |B_m|}$ for some absolute constant $C > 0$.

We will first need the following elementary lemma, which tells us that we can estimate $\omega(x,y)$ by the area of the triangle defined by $x$ and $y$.

\begin{lemma} \label{symplectic-projection}
  Let $x,y \in \ell_2$ be two vectors and let $\theta$ be their exterior angle.  Then
  \begin{align}
    |\omega(x,y)| \leq \|x\|\|y\| |\sin\theta|. \label{symplectic-estimate}
  \end{align}
\end{lemma}

\begin{proof}
  Let $V$ denote the 2-dimensional subspace spanned by $x$ and $ix$, and let $P : \ell_2 \to V$ denote the orthogonal projection onto $V$.  Then
  \begin{align*}
    \omega(x,y) = \omega(x,P(y)).
  \end{align*}
  If $x$ and $P(y)$ lie in a one-dimensional subspace, then $\omega(x,y) = 0$ and there is nothing to prove.  Thus, we may suppose $x$ and $P(y)$ span all of $V$.  If we define $W$ as the subspace in $\ell_2$ spanned by $x$ and $y$, we then have that $P|_W$ is an isomorphism.  It is also 1-Lipschitz as it is the restriction of an orthogonal projection.  As $\|x\|\|y\| |\sin \theta|$ is the area of the parallelogram $Q$ in $W$ with edges $x$ and $y$, we see that $|P(Q)| \leq |Q|$.  The lemma then follows once we see that $|P(Q)| = |\omega(x,y)|$.
\end{proof}

We now define the Koranyi norm on $\H_\infty$ analogously as before
\begin{align*}
  N(x,s) = \left(\|x\|^4 + s^2\right)^{1/4},
\end{align*}
where $\|\cdot\|$ is the standard $\ell_2$ norm and define the Koranyi metric as $d(x,y) = N(x^{-1}y)$.  

Note that the normal 3-dimensional Heisenberg group $\H$ equipped with its Koranyi norm embeds isometrically into $\H_\infty$ by
$$(x,y,z) \mapsto (x+iy,0,0,...,z).$$
Thus, the Laakso graphs embed into $\H_\infty$ with power $1/4$ and so $\H_\infty$ is not Markov $p$-convex for any $p < 4$.  We now let $\pi : \H_\infty \to \ell_2$ denote the homomorphic projection to $\ell_2$.  

We will follow the notation and terminology of \cite{matousek} and say that $P_n$ is the metric space $(\{1,...,n\},d_\Z)$.  Recall that $(X,d)$ is said to be $D$-biLipschitz equivalent to $(Y,\rho)$ if there exists some bijection $f : X \to Y$ and $s > 0$ so that
\begin{align*}
  s \cdot d_X(x,y) \leq d_Y(f(x),f(y)) \leq Ds \cdot d_X(x,y), \qquad \forall x,y \in X.
\end{align*}
A $\delta$-fork in $\H_\infty$ is a set $\{z_0,z_1,z_2,z_2'\}$ such that $\{z_0,z_1,z_2\}$ and $\{z_0,z_1,z_2'\}$ are both $(1+\delta)$-biLipschitz to $P_3$.  The following lemma tells us that if we have points in $\H_\infty$ that are $(1+\delta)$-biLipschitz to $P_3$, then they must be very straight and flat.

\begin{lemma} \label{small-angle}
  Let $\delta \in \left(0,10^{-100}\right)$ and $z = (x,s)$, $z'=(y,t)$ be elements in $\H_\infty$ such that $\{z,0,z'\}$ is $(1+\delta)$-biLipschitz to $P_3$.  Let
  \begin{align*}
    \eta := \frac{NH(z)}{N(z)}, \qquad  \nu := \frac{NH(z')}{N(z')},
  \end{align*}
  and $\theta$ be the exterior angle between $x$ and $y$ in $\ell_2$.  Then
  \begin{align*}
    |\theta| \leq 400\delta^{1/2}, \qquad \eta \leq 20\delta^{1/4}, \qquad \nu \leq 20\delta^{1/4}.
  \end{align*}
\end{lemma}

\begin{proof}
  As all the quantities $\eta,\nu,\theta$ do not change under dilation, we may suppose without loss of generality that $N(z) = 1$ and so as $\delta \leq 10^{-100}$, we have that
  \begin{align}
    N(z') &\in [(1-\delta)^2,(1+\delta)^2] \notag \\
    d(z,z') &\in [2(1-\delta)^2,2(1+\delta)^2]. \label{delta-ineq}
  \end{align}
  Note then that
  \begin{align}
    \|x\| &= N(z)(1-\eta^4)^{1/4} = (1-\eta^4)^{1/4}, \notag \\
    \|y\| &= N(z')(1-\nu^4)^{1/4} \leq (1+\delta)^2, \label{x-y-bounds}
  \end{align}
  as well as
  \begin{align}
    |s| &= \eta^2 N(z)^2 = \eta^2, \notag \\
    |t| &= \nu^2 N(z')^2 \leq (1+\delta)^4 \nu^2. \label{s-t-bounds}
  \end{align}

  {\bf Case 1:} $\theta > \pi/2$.

  We get by the law of cosines that
  \begin{align*}
    d(z,z')^4 &\overset{\eqref{symplectic-estimate}}{\leq} \|x-y\|^4 + \left( |s| + |t| + \frac{1}{2} \|x\| \|y\| |\sin \theta| \right)^2 \\
    &\leq 4(1+\delta)^8 + \left( \eta^2 + (1+\delta)^4 \nu^2 + \frac{1}{2} (1+\delta) |\sin \theta| \right)^2.
  \end{align*}
  Here, we've used the fact that if the interior angle of $x$ and $y$ is less than $\frac{\pi}{2}$, then $\|x-y\|$ cannot be larger than $\max\{\|x\|,\|y\|\} \sqrt{2 + 2\cos \frac{\pi}{2}} \leq (1+\delta)^2 \sqrt{2}$.  Continuing, we get
  \begin{multline*}
    d(z,z')^4 \leq (1+\delta)^8 \left[ 4 + \left( \eta^2 + \nu^2 + \frac{1}{2} |\sin \theta| \right)^2 \right] \\
    \leq (1+\delta)^8 \left[ 4 + \left( 2 + \frac{1}{2} |\sin \theta| \right)^2 \right]
  \end{multline*}
  As $\delta \leq 10^{-100}$, one sees that
  \begin{align*}
    d(z,z')^4 \leq (1+\delta)^8 \left[ 4 + \left( 2 + \frac{1}{2} |\sin \theta| \right)^2 \right] < (2-2\delta)^4, \qquad \forall \theta \in \left( \frac{\pi}{2} ,\pi \right],
  \end{align*}
  a contradiction of \eqref{delta-ineq}.

  {\bf Case 2:} $\max\{\nu,\eta\} > \left( \frac{15}{16} \right)^{1/4}$

  We suppose without loss of generality that $\eta = \max\{\nu,\eta\} > \left( \frac{15}{16} \right)^{1/4}$.  Thus, we get that
  \begin{align*}
    d(z,z')^4 &\overset{\eqref{symplectic-estimate}}{\leq} \|x-y\|^4 + \left( |s| + |t| + \frac{1}{2} \|x\| \|y\| \sin \theta \right)^2 \\
    &\overset{\eqref{s-t-bounds}}{\leq} (\|x\| + \|y\|)^4 + \left(\eta^2 + (1+\delta)^4 \eta^2 + \frac{1}{2} (1+\delta)(1-\eta^2)^{1/4} \right)^2 \\
    &\overset{\eqref{x-y-bounds}}{\leq} \left((1-\eta^4)^{1/4} + (1+\delta)^2\right)^4 + (1+\delta)^8 \left(2\eta^2 + \frac{1}{2} (1-\eta^4)^{1/4} \right)^2 \\
    &\leq (1+\delta)^8 \left[ \left((1-\eta^4)^{1/4} + 1\right)^4 + \left(2\eta^2 + \frac{1}{2}(1-\eta^4)^{1/4} \right)^2 \right].
  \end{align*}
  Note that $(1- \eta^4)^{1/4} \leq 1/2$.  As $\delta \leq 10^{-100}$, one gets for all $\eta \in \left( \left( \frac{15}{16} \right)^{1/4},1\right]$ that
  \begin{multline*}
    d(z,z')^4 \leq (1+\delta)^8 \left[ \left((1-\eta^4)^{1/4} + 1\right)^4 + \left(2\eta^2 + \frac{1}{2}(1-\eta^4)^{1/4} \right)^2 \right] \\
    < (2-2\delta)^4,
  \end{multline*}
  a contradiction of \eqref{delta-ineq}.

  {\bf Case 3:} $\theta \leq \pi/2$ and $\max \{\nu,\eta\} \leq \left( \frac{15}{16} \right)^{1/4}$.

  Note that we have proven that this is the only valid case, that is, if $\{z,0,z'\}$ is $(1+\delta)$-biLipschitz to $P_3$, then the exterior angle has to be less than $\pi/2$ and $\nu$ and $\eta$ cannot be too large.

  As $\delta \leq 10^{-100}$, we then get from the fact that $x \mapsto x^q$ is concave whenever $q \in [0,1]$ that
  \begin{align}
    \|x\|^{4q} &= (1-s^2)^q \leq 1 - qs^2, \notag \\
    \|y\|^{4q} &\overset{\eqref{delta-ineq}}{\leq} ((1+\delta)^4-t^2)^q \leq (1+5\delta-t^2)^q  \notag \\
    &\quad\leq (1-t^2)^q + \frac{q}{(1-t^2)^{1-q}} 10\delta \overset{\eqref{s-t-bounds}}{\leq} 1-qt^2 + 1000q\delta. \label{x-y-qbounds}
  \end{align}
  Here we used the fact that $|t| \leq (1+\delta)^2 \nu^2$.  By looking at the formula for the Koranyi norm, we have
  \begin{align}
    d(z,z')^4 \overset{\eqref{symplectic-estimate} \wedge \eqref{x-y-bounds}}{\leq} \|x - y\|^4 + \left( |t| + |s| + \frac{1}{2}(1+\delta) \sin\theta \right)^2. \label{d-x-x'}
  \end{align}

  By the law of cosines, we have that
  \begin{align}
    \|x - y\|^4 &= \left( \|x\|^2 + \|y\|^2 + 2\|x\|\|y\| \cos \theta \right)^2 \notag \\
    &= \|x\|^4 + \|y\|^4 + \|x\|^2\|y\|^2 (2 + 4 \cos^2 \theta) \notag \\
    &\qquad+ 4( \|x\|^3\|y\| + \|x\|\|y\|^3) \cos\theta \notag \\
    &\overset{\eqref{x-y-qbounds}}{\leq} 2+5\delta-t^2 -s^2 + \|x\|^2\|y\|^2 (2 + 4 \cos^2 \theta) \notag \\
    &\qquad+ 4 ( \|x\|^3\|y\| + \|x\|\|y\|^3 ) \cos\theta \label{e:xy-more-optimal} \\
    &\overset{\eqref{x-y-qbounds}}{\leq} 2+5\delta-t^2 -s^2 + \|x\|^2\|y\|^2 (2 + 4 \cos^2 \theta) + 4(2+2\delta) \cos\theta \label{e:xy-more-optimal-2} \\
    &\leq 10+13\delta-t^2 -s^2 - 2\theta^2 + \|x\|^2\|y\|^2 (2 + 4 \cos^2 \theta) \label{use-concavity} \\
    &\overset{\eqref{x-y-qbounds}}{\leq} 10+13\delta-t^2 -s^2 - 2\theta^2 \notag\\
    &\qquad+ \left( 1- \frac{s^2}{2} \right) \left( 1 - \frac{t^2}{2} + 500\delta\right) (2 + 4 \cos^2 \theta) \notag \\
    &\leq 10+ 13\delta-t^2 -s^2 - 2\theta^2 \notag \\
    &\qquad+ \left( 1- \frac{s^2}{2} - \frac{t^2}{2} + \frac{s^2t^2}{4} + 500\delta\right) (6 -  \theta^2) \label{use-cos} \\
    &\leq 16+4000\delta- 4t^2 -4s^2 - \left( 3 - \frac{s^2+t^2}{2} \right) \theta^2 + \frac{3}{2} s^2t^2. \notag
  \end{align}
  In \eqref{e:xy-more-optimal}, we used \eqref{x-y-qbounds} for $q = 1$ to bound
  \begin{align*}
    \|x\|^4 + \|y\|^4 \leq 1- s^2 + (1+\delta)^4 - t^2 \leq 2 + 5\delta - s^2 - t^2.
  \end{align*}
  Similarly, in \eqref{e:xy-more-optimal-2}, we bounded
  \begin{align*}
    \|x\|^3 \|y\| + \|x\| \|y\|^3 \leq 1 \cdot ((1 + \delta)^4 - t^2)^{1/4} + 1^3 \cdot ((1+\delta)^4 - t^2)^{3/4} \leq 2 + 2\delta.
  \end{align*}
  In \eqref{use-concavity}, we used the fact that $\cos \theta \leq 1 - \frac{\theta^2}{4}$ for $\theta \in [0,\frac{\pi}{2}]$.  In \eqref{use-cos}, we used the fact that
  \begin{align*}
    2 + 4 \cos^2 \theta \leq 6 - \theta^2, \qquad \forall \theta \in \left[-\frac{\pi}{2}, \frac{\pi}{2} \right].
  \end{align*}

  Together with \eqref{d-x-x'}, we get that
  \begin{align*}
    d(z,z')^4 &\leq 16+ 4000\delta- 4t^2 -4s^2 - \left(3 - \frac{s^2+t^2}{2} \right)\theta^2 + \frac{3}{2} s^2t^2 \\
    &\qquad+ \left( |t| + |s| + \frac{1}{2} (1+\delta) \sin \theta \right)^2 \\
    &\leq 16+ 4000\delta-4 t^2 -4 s^2 - \left(3 - \frac{s^2 + t^2}{2} \right)\theta^2 + \frac{3}{2}s^2t^2 \\
    &\qquad+ \left( |t| + |s| + \frac{2}{3} |\theta| \right)^2 \\
    &\leq 16+4000\delta-4 t^2 -4 s^2 - \left(3 - \frac{s^2 + t^2}{2} \right)\theta^2 + \frac{3}{2}s^2t^2 \\
    &\qquad+ 3\left( t^2 + s^2 + \frac{4}{9} \theta^2 \right) \\
    &\leq 16+4000\delta-t^2 -s^2 - \left( 3 - \frac{s^2 + t^2}{2}  - \frac{4}{3} \right) \theta^2 + \frac{3}{2}s^2t^2 \\
    &= 16+4000\delta- \left(1 - \frac{3}{4}s^2 \right) t^2 - \left(1 - \frac{3}{4} t^2 \right) s^2 \\
    &\qquad- \left( 3 - \frac{s^2 + t^2}{2}  - \frac{4}{3} \right)\theta^2.
  \end{align*}
  As $s^2 \leq \eta^4 \leq 1$ and $t^2 \leq (1+\delta)^8 \nu^4 \leq (1+\delta)^8$, and so we get that
  \begin{align*}
    d(z,z') \leq \left( 16+ 4000\delta- \frac{1}{5} ( t^2 + s^2 + \theta^2 ) \right)^{1/4}
  \end{align*}
  Thus, as $\eta = |s|^{1/2}$ and $\nu \leq \frac{|t|^{1/2}}{(1-\delta)^2}$, we see that if the conclusion of the lemma are not satisfied (that is $\eta^4$, $\nu^4$, or $\theta^2$ are large compared to $\delta$), then
  \begin{align*}
    d(z,z') < 2-2\delta,
  \end{align*}
  a contradiction of \eqref{delta-ineq}.
\end{proof}

We can then prove the following lemma, which says that if we have a $\delta$-fork in $\H_\infty$ and the tips are not too non-horizontal, then the tips actually collapse by a factor of $\delta^{1/2}$.

\begin{lemma} \label{collapse-convex}
  If $\delta \in (0,10^{-100})$ and $\{z_0,z_1,z_2,z_2'\}$ is a $\delta$-fork in $\H_\infty$ such that if $NH(z_2^{-1}z_2') < \frac{1}{2} d(z_2,z_2')$, then
  \begin{align*}
    d(z_2,z_2') \leq 2000 \delta^{1/2} \cdot d(z_0,z_1).
  \end{align*}
\end{lemma}

\begin{proof}
  We may suppose $z_1 = 0$ and $d(z_0,z_1) = 1$.  Let $z_2 = (x_2,t_2)$ and $z_2' = (x_2',t_2')$.  As
  \begin{multline*}
    \left(t_2 - t_2' + \frac{1}{2} \omega(\pi_1(x_2),\pi_1(x_2')) \right)^2 = NH(z_2^{-1}z_2')^4 \leq \frac{1}{16} d(z_2,z_2')^4 \\
    = \frac{1}{16} \|x_2 - x_2'\|^4 + \frac{1}{16} \left( t_2 - t_2' + \frac{1}{2} \omega(\pi_1(x_2),\pi_1(x_2')) \right)^2.
  \end{multline*}
  This tells us that (with some non-optimal estimates)
  \begin{align*}
    \left| t_2 - t_2' + \frac{1}{2} \omega(\pi_1(x_2),\pi_1(x_2')) \right|^2 \leq 15 \|x_2 - x_2'\|^4,
  \end{align*}
  and so
  \begin{align}
    d(z_2,z_2') \leq 2 \|x_2 - x_2'\|. \label{hilbert-d-bound}
  \end{align}
  From Lemma \ref{small-angle}, and the fact that $\{z_0,0,z_2\}$ and $\{z_0,0,z_2'\}$ are both $(1+\delta)$-biLipschitz to $P_3$, we know that the exterior angles $x_2$ and $x_2'$ make with the line spanned by $0$ and $x_0$ are less than $400 \delta^{1/2}$.  Thus,
  \begin{align}
    \angle x_2 0 x_2' < 800\delta^{1/2}. \label{delta2-angle}
  \end{align}
  If we set $\eta = \frac{NH(z_2)}{N(z_2)}$ and $\nu = \frac{NH(z_2')}{N(z_2')}$, we also know that
  \begin{align*}
    |t_2| = NH(z_2)^2 = \eta^2 N(z_2)^2 \leq 400(1 + \delta)^2 \delta^{1/2}.
  \end{align*}
  The same bound holds for $|t_2'|$.  Thus,
  \begin{align*}
    \|x_2\| = (N(z_2)^4 - t_2^2)^{1/4} \in (1- 10^{10} \delta, 1 + 10^{10}\delta)
  \end{align*}
  by a first order approximation and the fact that $N(z_2) \in (1-\delta,1+\delta)$.  Again, the same conditions hold for $\|x_2'\|$.  Suppose without loss of generality that $\|x_2\| \leq \|x_2'\|$.  Let $y = \|x_2\| \frac{x_2'}{\|x_2'\|}$.  Then
  \begin{multline*}
    \|x_2 - x_2'\| \leq \|x_2 - y\| + \|y - x_2'\| \overset{\eqref{delta2-angle}}{<} (1 + 10^{10} \delta) 800\delta^{1/2} + 10^{10} \delta \\
    \leq 1000 \delta^{1/2}.
  \end{multline*}
  This, together with \eqref{hilbert-d-bound}, proves the statement.
\end{proof}

We recall some more notation from \cite{matousek}.  The complete $k$-ary tree of depth $h$ is $T_{k,h}$.  As shown in \cite{matousek}, $T_{k,h}$ can be embedded into $B_{2h \lceil \log_2 k \rceil}$ with distortion at most 2.  For a rooted tree $T$, let $SP(T)$ denote the set of all unordered pairs $\{x,y\}$ of vertices of $T$ such that $x$ lies on the path from $y$ to the root.  The following Ramsey-type lemma is Lemma 5 from \cite{matousek}.

\begin{lemma}
  Let $h$ and $r$ be given natural numbers, and suppose that $k \geq r^{(h+1)^2}$.  Suppose that each of the pairs from $SP(T_{k,h})$ is colored by one of $r$ colors.  Then there exists a biLipschitz copy $T'$ of $B_h$ in this $T_{k,h}$ such that the color of any pair $\{x,y\} \in SP(T')$ only depends on the level of $x$ and $y$.
\end{lemma}

\begin{lemma}[Modified path embedding lemma] \label{l:metric-diff}
  For any $\alpha > 0$, there exists a constant $A = A(\alpha)$ with the following property.  Whenever $k \in \N$ and $f$ is a noncontracting mapping of the metric space $P_h$ into some other metric space $(X,d)$ so that $h \geq 2^{A\|f\|_{lip}^\alpha + k}$, then there exists a subspace $Z = \{x, x+\ell, x + 2\ell\} \subseteq P_h$ such that $\ell \geq 2^k$ and if we denote by $f_0$ the restriction of $f$ on $Z$, then $f_0$ is biLipschitz of distortion at most $1+\epsilon$ with
  \begin{align*}
    \epsilon = 10^{-100} \left( \frac{\ell}{d(f(x),f(x+\ell))} \right)^\alpha.
  \end{align*}
\end{lemma}

\begin{proof}
  The proof is almost exactly the same as the proof of Lemma 6 of \cite{matousek}, which we assume the reader is familiar with.  Here, we've fixed $\beta = 10^{-100}$.  The fact that we start with $h \geq 2^{A\|f\|_{lip}^\alpha + k}$ allows us to ensure that the two consecutive values of $K(2^i)$ and $K(2^{i+1})$ that lie in the same interval $[x_{j+1},x_j)$ can be chosen so that $i \geq k$.
\end{proof}

The next lemma says that, given sufficiently many vectors in $\ell_2$ of bounded length, there must be two vectors with small symplectic value.

\begin{lemma} \label{symplectic-collapse}
  There exists $\ell_0 > 0$ so that if $\ell \geq \ell_0$ and $\{z_i\}_{i=1}^N$ is a set of vectors in $\ell_2$ for which
  \begin{align*}
    N &\geq \frac{2^{\ell/2}}{16 \log \ell}, \\
    \|z_i\| &\leq \ell (\log \ell)^{1/2}, \qquad \forall i \in \{1,...,2^\ell\}.
  \end{align*}
  Then there exists $i \neq j$ so that
  \begin{align*}
    |\omega(z_i,z_j)| \leq \frac{1}{4} \ell^2.
  \end{align*}
\end{lemma}

\begin{proof}
  Suppose the claim is false.  Let $A_1 = \{z_i\}_{i=1}^N$.  Choose some $v_1 \in A_1$ and let $V_1$ be the 2-dimensional subspace in $\ell_2$ spanned by $v_1$ and $iv_1$.  Let $P_1 : \ell_2 \to \ell_2$ denote the orthogonal projection onto $V_1$.  If $\|P_1(u)\| < \frac{\ell}{5(\log \ell)^{1/2}}$ for some $u \in A_1$, then
  \begin{align*}
    |\omega(v_1,u)| \leq \|v_1\|\|P_1(u)\| \leq \ell (\log \ell)^{1/2} \|P_1(u)\| < \frac{1}{5} \ell^2,
  \end{align*}
  and so we would reach a contradiction.  Thus, we may assume that 
  \begin{align*}
    \|P_1(u)\| \geq \frac{\ell}{5(\log \ell)^{1/2}}, \qquad \forall u \in A_1.
  \end{align*}

  We divide $S^1$ into intervals of length $(\log \ell)^{-4}$ and group the vectors $u \in A_1$ by which interval the angle $P_1(u)$ makes with $P_1(v_1)$ falls into, breaking ties arbitrarily.  One of these intervals must have at least
  \begin{align*}
    \frac{N - 1}{(\log \ell)^4}
  \end{align*}
  vectors associated to it, which we will call $A_2$.  Let $Q_1$ denote the orthogonal projection onto $V_1^\perp$.  Choose a vector $v_2$ from $A_2$, let $V_2$ be the 2-dimensional subspace in $\ell_2$ spanned by $Q_1(v_2)$ and $iQ_1(v_2)$, and let $P_2 : \ell_2 \to \ell_2$ be the orthogonal projection onto $V_2$.  Then $V_2 \subset V_1^\perp$.  Note then that for each $u,v \in A_2$, we have that
  \begin{multline*}
    |\omega(u,v)| = |\omega(P_1(u),P_1(v)) + \omega(Q_1(u),Q_1(v))| \\
    \leq \frac{\ell^2}{(\log \ell)^3} + |\omega(Q_1(u),Q_1(v))|.
  \end{multline*}
  Note that if $u \in A_2$, then
  $$|\omega(Q_1(u),Q_1(v_2))| = |\omega(Q_1(u),P_2(v_2))| = |\omega(P_2(u),P_2(v_2))|.$$
  This gives us that
  \begin{multline*}
    |\omega(u,v_2)| \leq \frac{\ell^2}{(\log \ell)^3} + |\omega(P_2(u),P_2(v_2))| \\
    \leq \frac{\ell^2}{(\log \ell)^3} + \ell (\log \ell)^{1/2} \|P_2(u)\|.
  \end{multline*}
  We see as before that we must have $\|P_2(u)\| \geq \frac{\ell}{5(\log \ell)^{1/2}}$ for all $u \in A_2 \backslash \{v_2\}$ as otherwise we would have a contradiction if $\ell$ is large enough.
 
  We again divide up $S^1$ into intervals of length $(\log \ell)^{-4}$ and group the vectors $u \in A_2$ by which interval the angle $P_2(u)$ makes with $P_2(v_2)$ falls into.  One of these intervals must have at least
  \begin{align*}
    \frac{N - 1 - (\log \ell)^4}{(\log \ell)^8}
  \end{align*}
  vectors assigned to it, which we will take to be $A_3$.

  Continuing this way, we see that up to $k = 50(\log \ell)^2$, we can construct orthogonal symplectic subspaces $V_1,...,V_k$ and a subset of vectors $A_{k+1}$ for which
  \begin{align*}
    |A_{k+1}| \geq \frac{N - \sum_{j=0}^{k-1} (\log \ell)^{4j}}{(\log \ell)^{4k}} \geq \frac{N- k (\log \ell)^{4k}}{(\log \ell)^{4k}}
  \end{align*}
  such that if $u \in A_{k+1}$, then $\|P_{V_j}(u)\| \geq \frac{\ell}{5(\log \ell)^{1/2}}$ for every $j$.  By our choice of $N$ and $k$, if $\ell$ is larger than some absolute constant, then $A_{k+1}$ is non-empty.  But if $u \in A_{k+1}$, we have
  \begin{align*}
    \|u\|^2 \geq \sum_{j=1}^k \|P_{V_j}(u)\|^2 \geq 50(\log \ell)^2 \frac{\ell^2}{25 \log \ell} \geq 2 \ell^2 (\log \ell).
  \end{align*}
  This contradicts our assumption that $\|u\| \leq \ell (\log \ell)^{1/2}$.
\end{proof}

Clearly, the proof of Lemma \ref{symplectic-collapse} works for more general $N$, such as any exponent of $\ell$.  This $N$ will be the specific one we need.

\begin{proof}[Proof of Theorem \ref{tree-nonembedding}]
  In this proof, a familiarity with \cite{matousek} with be helpful (but not crucial) for the reader.  Suppose there exists a noncontracting map $f : B_m \to \H_\infty$ such that $\|f\|_{lip} = K = 2c(\log m)^{1/2}$ for $c > 0$ small enough so that if we set $h = 2^{(A(2)+1)K^2}$ then $h < m^{1/4}$.  Here we've applied Lemma \ref{l:metric-diff} to get $A(2)$.  If we also set $r = 10^{100} \cdot 2K^3$, and $k = r^{(h+1)^2} \leq (10^{100} K)^{4m^{1/2}} \leq \exp(Cm^{1/2} \log \log m)$, then
  $$2 h \log k < 2Cm^{3/4} \log \log m \leq m$$
  as long as $m$ is sufficiently large and so there exists a biLipschitz copy of $T_{k,h}$ inside $B_m$, and we can assume the map of $T_{k,h}$ into $B_m$ is noncontracting.  Let us restrict $f$ to this subtree.  It is clear that $f$ is still noncontracting and has the same Lipschitz bound.  If we color each pair $\{x,y\} \in SP(T_{k,h})$ according to the distortion of their distance by $f$
  \begin{align*}
    \left\lfloor 10^{100} K^2 \frac{d(f(x),f(y))}{d(x,y)} \right\rfloor \in \{0,...,r-1\},
  \end{align*}
  then by the fact that $k = r^{(h+1)^2}$, we get that there exists some subtree $B_h$ of $T_{k,h}$ such that the colors of $\{x,y\} \in SP(B_h)$ depend only on the levels of $x$ and $y$.

  Consider a root-leaf path $P$ in $B_h$.  As $h = 2^{(A(2)+1)K^2}$, there exists three vertices $x_0,x_1,x_2$ in $P$ at levels $j,j+\ell,j+2\ell$, respectively, such that $\ell \geq 2^{K^2}$ and $f$ is $(1+\delta)$-biLipschitz restricted on $x_0,x_1,x_2$ where
  \begin{align}
    \delta = 10^{-100} \left( \frac{\ell}{d(f(x_0),f(x_1))} \right)^2 \leq 10^{-100}. \label{delta-defn}
  \end{align}
  The latter inequality comes from the fact that $f$ is noncontracting.  We will suppose without loss of generality that $\ell$ is even.  Note that
  \begin{align}
    \log \ell \geq K^2 = 4c^2 \log m. \label{l-m-ineq}
  \end{align}
  
  Now consider all the descendents of $x_1$ in $B_h$ that lie $\ell/2$ levels down.  We can write them as such $\{x_i'\}_{i=1}^{2^{\ell/2}}$.  For each $i \in \{1,...,2^{\ell/2}\}$, choose a decendent of $x_i'$ in $B_h$ that lies on the same level as $x_2$.  Thus, we have chosen $2^{\ell/2}$ points and we denote them by $\{y_i\}_{i=1}^{2^{\ell/2}}$.  Note then that $\{x_0,x_1,y_i,y_j\}$ is a $\delta$-fork for each $i,j$.  Furthermore, we have that $\ell \leq d_T(y_i,y_j) \leq 2\ell$.

  We will suppose without loss of generality that $f(x_1) = 0$.  Consider the central coordinates $z_i$ of $f(y_i)$.  As we have for every $i \in \{1,...,2^{\ell/2}\}$ that
  \begin{align*}
    |z_i| \leq N(f(y_i))^2 = d(f(x_1),f(y_i))^2 \leq 4c^2\ell^2 (\log m) \overset{\eqref{l-m-ineq}}{\leq} \ell^2 (\log \ell),
  \end{align*}
  we get by the pigeonhole principle that there exists a subset $\{y_i'\}_{i=1}^N \subseteq \{y_i\}_{i=1}^{2^{\ell/2}}$ where $N \geq \frac{2^{\ell/2}}{16\log \ell}$ and all central coordinates of $f(y_i')$ differ by no more than $\frac{1}{16} \ell^2$.  We also have that
  \begin{align*}
    \|\pi(y_i')\| \leq N(f(y_i')) \leq d(f(x_1),f(y_i')) \leq \ell (\log \ell)^{1/2},
  \end{align*}
  and so applying Lemma \ref{symplectic-collapse} to $\{y_i'\}_{i=1}^N$, we get that there exist two elements (we will suppose by renaming them that they are 0 and 1) $y_0' = (a_0,b_0)$ and $y_1' = (a_1,b_1)$ so that
  \begin{align*}
    |\omega(a_0,a_1)| \leq \frac{1}{4} \ell^2.
  \end{align*}
  Thus, we see that
  \begin{multline*}
    NH(f(y_0')^{-1}f(y_1'))^2 = \left| b_0 - b_1 - \frac{1}{2} \omega(a_0,a_1) \right| \leq |b_0 - b_1| + \frac{1}{2} |\omega(a_0,a_1)| \\
    \leq \frac{1}{8} \ell^2 + \frac{1}{8} \ell^2 \leq \frac{1}{4} \ell^2 \leq \frac{1}{4} d(f(y_0'),f(y_1'))^2.
  \end{multline*}
  The last inequality comes from the fact that $d_T(y_0',y_1') \geq \ell$ and $f$ is noncontracting.  Thus, by Lemma \ref{collapse-convex}, we have that
  \begin{align*}
    \frac{\ell}{2} \leq d(f(y_0'),f(y_1')) \leq 2000 \delta^{1/2} d(f(x_0),f(x_1)) \overset{\eqref{delta-defn}}{\leq} 2000 \cdot 10^{-50} \ell < \frac{\ell}{2},
  \end{align*}
  a contradiction.
\end{proof}


\begin{bibdiv}
\begin{biblist}

\bib{ball}{article}{
  author = {Ball, K.},
  title = {The Ribe programme},
  book = {
    series = {S\'{e}minaire Bourbaki},
    note = {expos\'{e} 1047},
  },
  year = {2012},
}

\bib{bourgain}{article}{
  author = {Bourgain, J.},
  title = {The metrical interpretation of superreflexivity in Banach spaces},
  journal = {Israel J. Math.},
  volume = {56},
  number = {2},
  year = {1986},
  pages = {222-230},
}

\bib{cheeger-kleiner-1}{article}{
  author = {Cheeger, J.},
  author = {Kleiner, B.},
  title = {Differentiating maps into $L^1$ and the geometry of $BV$ functions},
  journal = {Ann. Math.},
  volume = {171},
  number = {2},
  year = {2010},
  pages = {1347-1385},
}

\bib{cheeger-kleiner-2}{article}{
  author = {Cheeger, J.},
  author = {Kleiner, B.},
  title = {Metric differentiation, monotonicity, and maps to $L^1$},
  journal = {Invent. Math.},
  volume = {182},
  number = {2},
  year = {2010},
  pages = {335-370},
}

\bib{ckn}{article}{
  author = {Cheeger, J.},
  author = {Kleiner, B.},
  author = {Naor, A.},
  title = {Compression bounds for Lipschitz maps from the Heisenberg group to $L_1$},
  journal = {Acta Math.},
  volume = {207},
  number = {2},
  year = {2011},
  pages = {291-373},
}

\bib{cygan}{article}{
  author = {Cygan, H.},
  title = {Subadditivity of homogeneous norms on certain nilpotent Lie groups},
  journal = {Proc. Amer. Math. Soc.},
  volume = {83},
  number = {1},
  year = {1981},
  pages = {69-70},
}

\bib{enflo}{inproceedings}{
  author = {Enflo, P.},
  title = {Banach spaces which can be given an equivalent uniformly convex norm},
  booktitle = {Proceedings of the International Symposium on Partial Differential Equations and the Geometry of Normed Linear Spaces (Jerusalem, 1972)},
  volume = {14},
  pages = {281-288},
  year = {1973},
}

\bib{james-1}{article}{
  author = {James, R.C.},
  title = {Uniformly non-square Banach spaces},
  journal = {Ann. Math. (2)},
  volume = {80},
  pages = {542-550},
  year = {1964},
}

\bib{james-2}{article}{
  author = {James, R.C.},
  title = {Super-reflexive Banach spaces},
  journal = {Canad. J. Math.},
  volume = {24},
  pages = {896-904},
  year = {1972},
}

\bib{johnson-schechtman}{article}{
  author = {Johnson, W. B.},
  author = {Schechtman, G.},
  title = {Diamond graphs and super-reflexivity},
  journal = {J. Topol. Anal.},
  volume = {1},
  number = {2},
  pages = {177-189},
}

\bib{laakso}{article}{
  author = {Laakso, T.J.},
  title = {Plane with $A_\infty$-weighted metric not bi-Lipschitz embeddable to $\R^N$},
  journal = {Bull. London Math. Soc.},
  volume = {34},
  number = {6},
  pages = {667-676},
  year = {2002},
}

\bib{lafforgue-naor}{article}{
  author = {Lafforgue, V.},
  author = {Naor, A.},
  title = {Vertical versus horizontal Poincar\'{e} inequalities on the Heisenberg group},
  journal = {Israel J. Math.},
  year = {2014},
  volume = {203},
  number = {1},
  pages = {309-339},
}

\bib{lang-plaut}{article}{
  author = {Lang, U.},
  author = {Plaut, C.},
  title = {Bilipschitz embeddings of metric spaces into space forms},
  journal = {Geom. Dedicata},
  volume = {87},
  number = {1-3},
  pages = {285-307},
  year = {2001},
}

\bib{LNP}{article}{
  author = {Lee, J.},
  author = {Naor, A.},
  author = {Peres, Y.},
  title = {Trees and Markov convexity},
  journal = {Geom. Funct. Anal.},
  volume = {18},
  number = {5},
  year = {2009},
  pages = {1609-1659},
}

\bib{li-carnot}{article}{
  author = {Li, S.},
  title = {Coarse differentiation and quantitative nonembeddability of Carnot groups},
  journal = {J. Funct. Anal.},
  volume = {266},
  pages = {4616-4704},
  year = {2014},
}

\bib{li-schul-1}{article}{
  author = {Li, S.},
  author = {Schul, R.},
  title = {The traveling salesman problem in the Heisenberg group: upper bounding curvature},
  journal = {Trans. Amer. Math. Soc.},
  note = {To appear},
  year = {2013},
}

\bib{li-schul-2}{article}{
  author = {Li, S.},
  author = {Schul, R.},
  title = {An upper bound for the length of a traveling salesman path in the Heisenberg group},
  journal = {Rev. Mat. Iberoam.},
  note = {To appear},
  year = {2014},
}

\bib{matousek}{article}{
  author = {Matou\v{s}ek, J.},
  title = {On embedding trees into uniformly convex Banach spaces},
  journal = {Israel J. Math.},
  volume = {114},
  number = {1},
  pages = {221-237},
  year = {1999},
}

\bib{MN}{article}{
	Author = {Mendel, M.},
  Author = {Naor, A.},
	Journal = {J. Eur. Math. Soc.},
	Number = {1},
  Volume = {15},
	Pages = {287-337},
	Title = {Markov convexity and local rigidity of distorted metrics},
	Year = {2013}
}

\bib{montgomery}{book}{
  author = {Montgomery, R.},
  title = {A tour of sub-Riemannian geometries, their geodesics and applications},
  series = {Mathematical Surveys and Monographs},
  volume = {91},
  publisher = {American Mathematical Society},
  year = {2002},
}

\bib{naor}{article}{
  author = {Naor, A.},
  journal = {Jpn. J. Math.},
  title = {An introduction to the Ribe program},
  volume = {7},
  number = {2},
  pages = {167-233},
  year = {2012},
}

\bib{pansu}{article}{
  author = {Pansu, P.},
  journal = {Ann. Math.},
  title = {M\'etriques de Carnot-Carath\'eodory et quasiisom\'etries des espaces sym\'etriques de rang un},
  volume = {128},
  pages = {1-60},
  year = {1989},
}

\bib{pisier}{article}{
  author = {Pisier, G.},
  journal = {Israel J. Math.},
  title = {Martingales with values in uniformly convex spaces},
  volume = {20},
  number = {3-4},
  pages = {326-350},
  year = {1975},
}

\bib{ribe}{article}{
  author = {Ribe, M.},
  title = {On uniformly homeomorhic normed spaces},
  journal = {Ark. Mat.},
  volume = {14},
  number = {1-2},
  pages = {237-244},
  year = {1976},
}

\bib{semmes}{article}{
  author = {Semmes, S.},
  title = {On the nonexistence of bi-Lipschitz parameterizations and geometric problems about $A_\infty$-weights},
  journal = {Rev. Mat. Iberoam.},
  volume = {12},
  number = {2},
  pages = {337-410},
  year = {1996},
}


\end{biblist}
\end{bibdiv}
\end{document}